\documentclass[a4paper,reqno,10pt]{amsart}
\usepackage{newtxtext}
\usepackage{microtype}
\usepackage{fix-cm}
\usepackage[T1]{fontenc}
\usepackage[numbers]{natbib}  
\usepackage{textcomp}
\usepackage{amsfonts} 
\usepackage{amssymb}
\usepackage{amsthm}
\usepackage{amsmath} 
\usepackage{mathrsfs} 
\usepackage{dsfont}
\usepackage{esint}
\usepackage{algorithm}
\usepackage{soul}
\usepackage{algpseudocode}
\usepackage[english]{babel} 
\usepackage{xcolor}
\definecolor{gray}{named}{gray}

\usepackage[left=3.5cm,right=3.5cm,top=3.5cm,bottom=3cm,headsep=0.7cm]{geometry}
\input{insbox}
\usepackage{tikz}
\usepackage{pgf,pgfplots,wrapfig}
\usetikzlibrary{arrows}
 
\usepackage[scriptsize,bf]{caption}

\usepackage{enumerate}
\usepackage{enumitem}

\usepackage{hyperref}
\usepackage{cleveref}
\usepackage{mathtools}
\hypersetup{
    colorlinks,
    linkcolor={red!80!black},
    citecolor={blue!80!black}
}

\theoremstyle{plain}
\begingroup
\theoremstyle{plain}
\newtheorem{theorem}{Theorem}[section]
\newtheorem{corollary}[theorem]{Corollary}
\newtheorem{proposition}[theorem]{Proposition}
\newtheorem{assumption}[theorem]{Assumption}
\newtheorem{lemma}[theorem]{Lemma}
\theoremstyle{definition}
\newtheorem{definition}[theorem]{Definition}
\theoremstyle{remark}
\newtheorem{remark}[theorem]{Remark}
\newtheorem{example}[theorem]{Example}
\endgroup

\theoremstyle{definition}
\theoremstyle{remark}

\numberwithin{equation}{section}

\makeatletter
\newcommand{\myitem}[1]{%
\item[#1]\protected@edef\@currentlabel{#1}%
}
\makeatother

\def\ddj{\dot\Delta_j}

\newcommand{\R}{\mathbb{R}}
\newcommand{\N}{\mathbb{N}}
\newcommand{\Z}{\mathbb{Z}}
 
\newcommand{\M}{\mathbb{M}} 
\newcommand{\rank}{\mathrm{rank}}

\mathchardef\emptyset="001F

\renewcommand{\d}{{\mathfrak{d}}}
\renewcommand{\tilde}{\widetilde}

\renewcommand{\l}{\langle}
\renewcommand{\r}{\rangle}
\newcommand{\eps}{\varepsilon}

\renewcommand{\L}{\mathcal{L}} 
\renewcommand{\i}{{\mathrm{i}}}

\author[T.~Crin-Barat]{Timoth\'ee~Crin-Barat} 
\address[T.~Crin-Barat]{Universit\'e Paul Sabatier,  Institut de Math\'ematiques de Toulouse, Route de Narbonne 118, 31062 Toulouse Cedex 9, France.}
\email{timothee.crin-barat@math.univ-toulouse.fr}

\author[L. Liverani]{Lorenzo Liverani*}
\address[L. Liverani]{Friedrich-Alexander-Universit\"at Erlangen-N\"urnberg, Department of Data Science, Chair for Dynamics, Control and Numerics (Alexander von Humboldt Professorship), Cauerstr. 11, 91058 Erlangen, Germany. }
\email{lorenzo.liverani@fau.de}

\author[L-Y. Shou]{Ling-Yun Shou}
\address[L-Y. Shou]{School of Mathematical Sciences and Mathematical Institute, Nanjing Normal University, Nanjing, 210023,
P. R. China.}
\email{shoulingyun11@gmail.com}

\author[E. Zuazua]{Enrique Zuazua}
\address[E. Zuazua]{Friedrich-Alexander-Universit\"at Erlangen-N\"urnberg, Department of Data Science, Chair for Dynamics, Control and Numerics (Alexander von Humboldt Professorship), Cauerstr. 11, 91058 Erlangen, Germany. 
	\newline \indent 
	Chair of Computational Mathematics, Fundación Deusto,	Avenida de las Universidades, 24, 48007 Bilbao, Basque Country, Spain. 
	\newline \indent
	Universidad Autónoma de Madrid, Departamento de Matemáticas, Ciudad Universitaria de Cantoblanco, 28049 Madrid, Spain.}
	\email{enrique.zuazua@fau.de}

\thanks{*Corresponding author: lorenzo.liverani@fau.de}

\title[Asymptotics for Hyperbolic systems with non-symmetric relaxation]{Large-Time Asymptotics for Hyperbolic Systems with non-symmetric relaxation: An Algorithmic Approach
}

\pgfplotsset{compat=1.16}


\begin{document}

 \begin{abstract}
     We study the stability of one-dimensional linear hyperbolic systems with non-symmetric relaxation. Introducing a new frequency-dependent Kalman stability condition, we prove an abstract decay result underpinning a form of \textit{inhomogeneous hypocoercivity}. In contrast with the homogeneous setting, the decay rates depend on how the Kalman condition is fulfilled and, in most cases, a loss of derivative occurs: one must assume an additional regularity assumption on the initial data to ensure the decay.
 
    Under structural assumptions, we refine our abstract result by providing an algorithm, of wide applicability, for the construction of Lyapunov functionals. This allows us to systematically establish decay estimates for a given system and uncover algebraic cancellations (beyond the reach of the Kalman-based approach) reducing the loss of derivatives in high frequencies. To demonstrate the applicability of our method, we derive new stability results for the Sugimoto model, which describes the propagation of nonlinear acoustic waves, and for a beam model of Timoshenko type with memory.


 \end{abstract}

\maketitle

\section{Introduction}
\label{sec:1}
\noindent
In this work, we consider one-dimensional linear hyperbolic systems of the form
\begin{align}
\partial_t U + A\partial_x U + B U = 0,\label{eq:1-1}
\end{align}
ruling the evolution of the variable $U = U(t,x):\R^+ \times \R^n \to \R^n$. Here, $A \in \mathbb M(n,n)$ is a symmetric matrix, whereas $B\in \mathbb M(n,n)$ is allowed to be non-symmetric. Denoting by
\[
B^a = \frac{B - B^\top}{2}, \qquad B^s = \frac{B+B^\top}{2}
\]
its skew-symmetric and symmetric part respectively, the system \eqref{eq:1-1} can be reshaped into
\begin{equation}
    \label{eq:main}
\partial_t U +A\partial_x U + B^a U + B^s U = 0.
\end{equation}
Systems of the form \eqref{eq:main} are ubiquitous in physical modelling, where they can either serve as self-standing mathematical models or, perhaps more importantly, as the linearization of nonlinear systems describing some physical phenomenon. In this context, it is paramount to understand the asymptotic properties of the solution to \eqref{eq:main}, as these properties often translate into global-in-time well-posedness properties for the associated nonlinear system (see the forthcoming Remark \ref{rq:nonlin} and Example \ref{ex:sugi}).


Our goal is to provide a unifying framework for the large-time analysis of hyperbolic PDEs with non-symmetric relaxation when the system is endowed with a partially dissipative structure, namely when the matrix $B^s$ is of the form
\begin{equation}\label{B}
\begin{aligned}
B^s =
\begin{bmatrix}
   0 &0 \\
   0 &D
\end{bmatrix},
\end{aligned}
\end{equation}
where $D \in \M(n_2,n_2)$, $1\leq n_2<n$, is a positive definite symmetric matrix. Under these conditions, there exists a $\kappa>0$ (in this case the smallest eigenvalue of $D$) such that, for all $X \in \R^{n_2}$,
\begin{align}
\l D X, X\r \geq \kappa |X|^2.\label{partialdiss}
\end{align}
Due to the lack of coercivity of the dissipative operator $B^s$ and the conservative properties of the operator $A\partial_x+B^a$, standard $L^2$ energy estimates do not provide enough information to derive time-decay estimates for the solution. Indeed, using the symmetry of $A$ and the skew-symmetry of $B^a$, we obtain
\begin{align}
    \dfrac{1}{2}\dfrac{d}{dt}\|U\|_{L^2}^2+\|B^sU\|_{L^2}^2=0.
\end{align}
Nevertheless, in many scenarios, systems of the form \eqref{eq:main} exhibit spectral properties suggesting that the $L^2$ norm of the whole solution decays to zero in time. To justify such results with a priori estimates, one has to resort to hypocoercive approaches, which seek to uncover hidden damping mechanisms arising in the interactions of the dissipative and conservative parts of the system. 


\subsection{The symmetric case}
We recall some results regarding the stability of hyperbolic systems with symmetric relaxation, i.e. $B^a=0$. 
When $B^s$ is fully dissipative, i.e. $B^s>0$, the solutions decay exponentially in time \cite{LiTT}. In the partially dissipative scenario 
 \eqref{B}-\eqref{partialdiss}, the dissipation induced by $B^sU$ lacks coercivity, as it affects only the $n_2$ components of the solution. 
To analyze this situation, the celebrated papers \cite{Kawa1,SK} introduce the Shizuta-Kawashima (SK) stability condition: $(A,B^s)$ satisfy the (SK) condition if
\begin{align}   
\{\textrm{eigenvectors of A}\} \cap \textrm{Ker}(B^s)=\{0\}. \tag{SK}
\end{align}
Under this condition, time-decay estimates in the partially dissipative setting were obtained in \cite{Kawa1,SK,UKS}. Kawashima and Yong \cite{KY} then formulated a notion of entropy which plays a key role in symmetrizing quasilinear systems. Under the (SK) and entropy conditions, multiple studies were dedicated to show the global existence and stability of classical solutions for nonlinear hyperbolic systems with symmetric relaxation. For instance, Hanouzet and Natalini \cite{HanouzetNatalini} and Yong \cite{Yong} proved the global existence of classical solutions of the Cauchy problem for initial data close to equilibrium in Sobolev spaces. Ruggeri and Serre \cite{RS1} investigated the asymptotical $L^2$-stability of solutions around constant equilibrium states. Bianchini et al. \cite{BHN} carried out a detailed analysis of the Green function and established the $L^p$-decay rates and asymptotic stability of solutions with small perturbations.

More recently, Beauchard and Zuazua in \cite{BZ} showed the equivalence of the (SK) condition and the Kalman rank condition employed in control theory, namely: $(A,B^s)$ satisfy the Kalman rank condition if
\begin{align}\label{Kalman0}
\rank \left[
\begin{array}{c}
    B^s  \\
    B^s A \\
    \ldots \\
    B^s A^{n-1}
\end{array}\right] = n.
\end{align}
Under this condition, they proved the $L^2$-stability of system \eqref{eq:main} (with $B^a=0$) by constructing Lyapunov functionals, in the spirit of Villani's hypocoercivity theory developed in \cite{Villani}. For existence and decay results of nonlinear system building upon this approach, we refer to \cite{CBD3,CBD2,CBSZ,KYDecay,XK1,XK2} and references therein.



A basic example fitting the symmetric relaxation theory is given by the classical one-dimensional weakly damped wave equation
\[
\partial_{tt}u - \partial_{xx}u + \partial_tu = 0
\]
which can be rewritten in the form \eqref{eq:main} by introducing the variable $v =u_x$, obtaining
\[
\begin{dcases}
    \partial_t u + \partial_x v + u = 0, \\
    \partial_tv + \partial_x u = 0,
\end{dcases}
\]
so that 
\[
A = 
\left[\begin{array}{cc}
    0 & 1  \\
    1 & 0
\end{array}\right], \quad
B^s = 
\left[\begin{array}{cc}
    1 & 0  \\
    0 & 0
\end{array}  \right] \quad \text{and} \quad B^a=0.
\]
The system is partially dissipative, a property reflected by the energy equality
\[
\frac{1}{2}\frac{d}{dt}(|\widehat{u}|^2+|\widehat{v}|^2) + |\widehat{u}|^2 = 0, 
\]
where $\widehat{f}:=\mathcal{F}f$ denotes the Fourier transform of a tempered distribution $f$.
A quick computation shows that
\[
\rank \left[
\begin{array}{c}
    B^s  \\
    B^s A
\end{array}\right] = 2,
\]
so that the Kalman rank condition holds and the theory in \cite{BZ} can be applied. In this case, the Lyapunov functional leading to the decay rate reads
\begin{align}
    \mathcal{L}(t)=(|\widehat{u}|^2+|\widehat{v}|^2)+ \min\Big\{|\xi|,\frac{1}{|\xi|}\Big\} {\rm Im}\,\Big\l \widehat{u}, \frac{\xi}{|\xi|}\widehat{v} \Big\r
\end{align}
where $\langle \cdot,\cdot \rangle $ denotes the Hermitian
product on $\mathbb{C}$. Differentiating in time $\mathcal{L}$, one obtains
\begin{align}
  \dfrac{d}{dt}\mathcal{L}+c\min\{1,|\xi|^2\}(|\widehat{u}|^2+|\widehat{v}|^2)\leq 0,
\end{align} 
highlighting dissipative effects for $v$. Then, using that $\mathcal{L}\sim |\widehat{u}|^2+|\widehat{v}|^2$,
one recovers the decay estimates expected by the spectral analysis, i.e. a heat-like behaviour in low frequencies and exponential damping in high frequencies: for $(u_0,v_0)\in L^2\cap L^1$, there exist two constants $C>0$ and $\gamma^*>0$ such that
 \begin{align}\label{decay}
& \Vert (u,v)^\ell(t) \Vert_{L^{\infty}} \le C t^{-\frac{1}{2}} \Vert (u_0,v_0)\Vert_{L^1},\\ \label{decayhf}
&\Vert  (u,v)^h(t) \Vert_{L^2} \le C e^{-\gamma_{*} t} \Vert (u_0,v_0)\Vert_{L^2},
\end{align} where $(u,v)^\ell= (\widehat u,\widehat v) (t,\xi)\mathbf{1}_{|\xi|\leq 1}$ and $(u,v)^h= (\widehat u,\widehat v) (t,\xi)\mathbf{1}_{|\xi|\geq 1}$.
\medbreak
\noindent
For general systems with symmetric relaxation, the Lyapunov functional reads
\begin{align}\label{LyaSym}
\L(\xi,t)=\dfrac{1}{2}|U(\xi,t)|^2+\min\Big\{|\xi|,\frac{1}{|\xi|}\Big\}\sum_{k=1}^{n-1}\eps_k\textrm{Im}(\l B^sA_\omega^{k-1}\widehat U(\xi,t),B^sA^k_\omega\widehat U(\xi,t)\r),
\end{align}
where $A_\omega=A\omega$ for $\omega=\xi/|\xi|$.
For $\eps_k>0$ chosen small enough, one derives similar decay estimates as \eqref{decay}-\eqref{decayhf} when the Kalman rank condition \eqref{Kalman0} holds.
The structure of \eqref{LyaSym} underlines an iterative procedure designed to reveal the dissipative effect for each variable of the system. Indeed, each term added to the energy can be interpreted as the interaction between a variable whose dissipation has already been recovered and
another whose dissipation remains to be obtained. Accordingly, for $k=1$, the interaction between $B^s\widehat{U}$ and $B^sA\widehat{U}$ yields the dissipation of the latter, which is subsequently utilized in the next term to derive the dissipation of $B^sA^2\widehat{U}$. This process continues iteratively and naturally concludes at $k=n-1$, as the Cayley-Hamilton theorem ensures that no additional gain can be achieved beyond this point. Since $A\partial_x$ is a skew-symmetric operator, we can also interpret its action as a rotation in the space of functions, capable at each step of propagating the dissipation generated by $B^s$ to one of the $n - n_2$ non-damped (initially) directions.

\subsection{The skew-symmetric case -- State of the art}
\label{subsec12} Although the theory of symmetric relaxation can handle a wide range of physical models, many systems of interest fall outside the scope of the setting described above. A significant example is the Sugimoto system 
\begin{equation}
\label{sugimoto}
\begin{dcases}
    \partial_t u + \partial_x \left(a u + b \frac{u^2}{2}\right) +\Omega^2\varphi=0, \\
    \partial_{t}\varphi + \eps\varphi + \omega^2 p -u=0, \\ 
    \partial_t p = \varphi,
\end{dcases}
\end{equation}
proposed by Sugimoto in \cite{Sugimoto92} to model the propagation of nonlinear acoustic waves in a tube, with the ultimate goal of describing the high-amplitude waves generated by high-speed trains in a tunnel. 
Linearizing the system around the null solution, and rescaling the variables as $(u,\varphi,p) \mapsto (u, \Omega \varphi, \Omega \omega p)$, one obtains
\begin{equation}
\label{sugimoto_lin}
    \begin{dcases}
    \partial_t u + a\partial_x u +\Omega\varphi = 0, \\
    \partial_{t}\varphi - \Omega u + \eps\varphi + \omega p = 0, \\ 
    \partial_t p - \omega \varphi = 0,
\end{dcases}
\end{equation}
which is a system of the form \eqref{eq:main}, where
\[ A = 
\left[\begin{array}{ccc}
    a & 0 & 0  \\
    0 & 0 & 0 \\
    0 & 0 & 0
\end{array} \right], \quad B^a = 
\left[\begin{array}{ccc}
     0 & \Omega & 0\\
     -\Omega &0 & \omega \\
     0 &-\omega &0
\end{array} \right]\quad \text{and }\quad
B^s = 
\left[\begin{array}{ccc}
    0 & 0 &0 \\
    0 & \eps &0 \\
    0 &0 &0
\end{array} \right].
\]
In \cite{JuncaLombard20}, the finite-time formation of singularities and the existence of entropy solutions for system \eqref{sugimoto} was investigated. However, the asymptotic behavior was not addressed, leaving the question of whether a unique global-in-time strong solution exists for small initial data still open.
\medbreak
In fact, addressing systems with skew-symmetric dissipation remains a challenge. In the literature, the analysis often requires \emph{ad hoc} approaches to identify suitable energy-like functionals through the use of appropriate multipliers. Moreover, decay rates vary significantly depending on the algebraic structure of the systems, more precisely, on how the Kalman condition is satisfied. Typically, one observes weaker diffusion in the low-frequency regime and loss of derivatives in high frequencies.

We refer to the decay properties of the dissipative Timoshenko system \cite{IHK}, the Euler-Maxwell system \cite{UedaKawa2011,CBPSX}, the Timoshenko system with memory \cite{MoriMem}, the Timoshenko-Cattaneo system \cite{TimoCat}, the thermoelastic system of MGT-type \cite{PQU}, the 1D porous-elasticity system \cite{QUU} and the Bresse-Cattaneo system \cite{BresseCat}. In each of these cases, the matrix $B$ is non-symmetric and, in high frequencies, the real part of the eigenvalues of $(i\xi A +B)$ is asymptotically equal to $|\xi|^{-2\alpha}$ with $\alpha=0,1 \text{ or }2$, while in low frequencies they are asymptotically equal to $|\xi|^{2\beta}$ with $\beta=1 \text{ or }2$. Note that when $\alpha=0$ and $\beta=1$, we retrieve the decay rates of the symmetric case $B^a=0$.

The picture is not much different for hyperbolic systems with non-symmetric relaxation on bounded domains, and similar techniques can be applied to recover decay for high frequencies. We mention, without the claim of being exhaustive, the analysis of systems of Bresse and Timoshenko type \cite{DTimo, DPTimo}, as well as the investigation of thermoviscoelastic models \cite{MGTF1, Antitermo}. We also refer to \cite{AchleitnerArnoldMehrmannNighsch25} where the authors study hypocoercive properties (for short and large times) of linear evolution equations.
\smallbreak
While many studies analyzed specific examples of hyperbolic systems with non-symmetric relaxation, relatively few works have attempted to tackle the abstract problem in full generality. Two references in this context are  \cite{UDK1} and \cite{Mori}. In  \cite{UDK1}, Ueda, Duan and Kawashima formulated a new structural condition to analyze the particular type of regularity-loss mechanism when $\alpha=1$. In \cite{Mori}, Mori presented a (SK)-type mixed criterion that can be applied under restrictions on $\alpha$ and $\beta$, allowing to handle some cases where $\alpha=0,1$ and $\beta=1,2$.
\medbreak

Compared to the symmetric setting, where the two operators $A$ and $B^s$ can only interact in one way, the main difficulty in analyzing hyperbolic systems with non-symmetric relaxation lies in the presence of three operators, the two conservative ones $A\partial_x$ and $B^a$ and the dissipative one $B^s$. As mentioned above, the skew-symmetric operators $A\partial_x$ and $B^a$ can be interpreted as rotations acting on the initial dissipative operator $B^s$ and may propagate its dissipative effect to new directions. However, in contrast to the symmetric case, at each step of the iterative process, the two rotational effects of the two conservative operators can either be combined or performed independently, each scenario leading to different outcomes regarding the decay rates. This comes from the fact that the orders of the operators $A \partial_x$ and $B^a$ are different and require us to develop a form of \textit{inhomogeneous hypocoercivity} theory. 

\subsection{Our contributions} 

In the present paper, we develop a method to derive asymptotic decay results for $n\times n$ one-dimensional partially dissipative systems with non-symmetric relaxation, extending the approach proposed in \cite{BZ} dedicated to the symmetric case. Compared to the results of \cite{Mori,UDK1}, our structural assumptions on the matrices $A$ and $B$ are more general and our method can be applied to a wider variety of systems without restriction on the type of decay rates, including the linearized Sugimoto model \eqref{sugimoto_lin}, for which we can show the asymptotic decay of the solution. 
 
We begin by proving a general abstract result for systems with non-symmetric relaxation in Theorem \ref{thm:decayKalman}. Specifically, under a Kalman-type condition, we establish the large-time stability of solutions of the system \eqref{eq:main}. This algebraic Kalman condition can be seen as the natural extension of the one employed in \cite{BZ} and connects the three operators appearing in our system by considering the interactions between the conservative operator $A\partial_x + B^a$ and the dissipative operator $B^s$.

However, Theorem \ref{thm:decayKalman} does not provide a simple way to estimate the decay rates of a given system. 
For this reason, our second main result, Theorems \ref{th:improvedh}-\ref{th:improvedl}, improves upon it by refining the construction of the Lyapunov functional so that, in the end, a precise estimate for the decay rate of the solution can be easily computed. Moreover, this construction allows us to describe a phenomenon that the Kalman-based method developed in \ref{thm:decayKalman} cannot capture: it provides explicit algebraic conditions on the matrices $A$, $B^a$ and $B^s$ which, if satisfied, produce a cancellation reducing the loss of derivative in high frequencies. This phenomenon is well known for systems with non-symmetric relaxation, a paradigmatic example being the \emph{equal wave speed} condition for the Timoshenko system \cite{IHK}. Our new approach allows us to further apprehend this phenomenon for general hyperbolic systems. Nevertheless, we need to assume relatively strong structural conditions on the matrices to be able to justify this phenomenon. Relaxing these assumptions and improving our understanding of this cancellation mechanism is the subject of ongoing research.


The refinement of the Lyapunov functional is achieved through an algorithm, which we consider to be one of the main contributions of this work. Indeed, due to its simple and systematic nature, it can be easily applied to systems of any dimension, significantly reducing the effort of finding Lyapunov functionals for any given hyperbolic system with relaxation. We refer to Section \ref{sec:applications} for practical applications of our methodology and to Example \ref{ex:sugi} for a stability result concerning the linear Sugimoto model \eqref{sugimoto_lin}.






\subsection{Outline of the paper} In forthcoming Section \ref{sec:2}, we introduce the Inhomogeneous Kalman rank condition, and state our first main result, Theorem \ref{thm:decayKalman}. The proof of the latter is presented in Section \ref{sec:3}. In Section \ref{sec:4}, we introduce the algorithm used to construct the new Lyapunov functional. The construction is different depending on whether we work in high or low frequencies: Section \ref{sec:ImproveHF} focuses on the former and Section \ref{sec:ImproveLF} on the latter. In Section \ref{sec:7}, we state and prove our second result, divided into Theorem \ref{th:improvedh} and \ref{th:improvedl}. Section \ref{sec:8} is devoted to an in-depth analysis of the physically relevant and common case $\textrm{rank}(B^s)=1$. We then present some examples in Section \ref{sec:applications}, showing that our algorithm can be successfully applied to recover the optimal decay rates of some well-known physical systems. We draw some conclusions and discuss extensions in the final Section \ref{sec:10}. Technical results related to rank-one matrices are collected in Appendix \ref{sec:appendixB}, while the proof of Lemma \ref{lemma:equiv} is carried out in Appendix \ref{sec:appproof}.

\section{The Inhomogeneous Kalman Method}
\label{sec:2}
\noindent
In this section, we provide a Kalman-type analysis of the decay rates in the case of hyperbolic partially dissipative systems with non-symmetric relaxation terms. 

First, we introduce a Kalman condition adapted to studying \eqref{eq:main}. Applying the Fourier transform to \eqref{eq:main}, for every $\xi\in \R$, we obtain
\begin{equation}
    \label{eq:TFmain}
\partial_t \widehat{U} +\i \xi A\widehat{U}  + B^a \widehat{U}  + B^s \widehat{U}  = 0.
\end{equation}

\begin{definition}[Inhomogeneous Kalman rank condition]
\label{def:inhomo}
  Let $K\geq 1$.  We say that $(B^s,A,B^a)$ satisfies the inhomogeneous Kalman rank condition of order $K$  if, for every $\xi\in \R^*$,
\begin{align}\label{condition:span}
\rank \left[
\begin{array}{c}
    B^s  \\
    B^s (i\xi A+B^a) \\
    \ldots \\
    B^s (i\xi A+B^a)^{K}
\end{array}\right] = n.
\end{align}
\end{definition}
Due to the inhomogeneous aspect (in terms of $\xi$) of the above condition, we need the following lemma which relates the Kalman components $B^s(A\partial_x+B^a)^kU$ to the solution $U$.

\begin{lemma}
\label{lemma:equiv}
    Let the inhomogeneous Kalman condition \eqref{condition:span} hold. There exist integers $\alpha\in \N$ and $\beta \in \N$  such that 
\begin{enumerate}
    \item[(i)] for $|\xi|\geq 1$, 
    \begin{align}\label{decay:alpha}
        \sum_{k=0}^{K}|\xi|^{-2k}|B^s(\i \xi A+B^a)^k \widehat{U}|^2 \geq c |\xi|^{-2\alpha}| \widehat{U}|^2,
    \end{align}
    \item[(ii)] for $|\xi|\leq 1$,
    \begin{align}\label{decay:beta}
        \sum_{k=0}^{K}|B^s(\i \xi A+B^a)^k \widehat{U}|^2 \geq c |\xi|^{2\beta}| \widehat{U}|^2,
    \end{align}
\end{enumerate}
where $c>0$ is a universal constant.
\end{lemma}

The proof of this lemma is quite technical and provides little insight into the forthcoming discussion. For this reason, we moved it to the Appendix \ref{sec:appproof}.

\begin{remark}
The frequency weight $|\xi|^{-2k}$ appearing in \eqref{decay:alpha} is connected to the construction of the Lyapunov functional that we employ to justify the decay of the solutions in high frequencies. Notice that in the symmetric case $B^a=0$, the weight absorbs the $\xi$ parameter coming from the operator $A\partial_x$ and ensures that there is no loss of regularity, i.e., $\alpha=0$, recovering \eqref{decayhf}.\end{remark}

We are now in a position to state our first main result.
\begin{theorem}\label{thm:decayKalman}
   Let $K\geq 1$ and assume that $(B^s,A,B^a)$ satisfies the inhomogeneous Kalman rank condition of order $K$. Let $\alpha$ and $\beta$ be the smallest integers such that \eqref{decay:alpha} and \eqref{decay:beta} hold, and assume that the initial data satisfies $U_0\in L^1 \cap H^{\gamma \alpha}$ with $\gamma>0$. Then, there exists a constant $C>0$ such that
   \begin{enumerate}
        \item In high frequencies: if $\alpha\geq 1$, 
        the solution of \eqref{eq:main} satisfies
    \begin{align}
    \|U^h\|_{L^2} \leq C(1+t)^{-{\frac{\gamma}{2}}}\|U_0\|_{H^{\gamma\alpha}}
\end{align}
where  $U^{h}=\mathcal{F}^{-1}(\mathbf{1}_{|\xi|\geq 1}\widehat{U})$. 
If $\alpha=0$, then
\begin{align}
    \|U^h\|_{L^2} \leq Ce^{-t}\|U_0\|_{L^2}.
\end{align}
\item In low frequencies: if $\beta\geq 1$, the solution of \eqref{eq:main} satisfies
\begin{align}
    \|U^\ell\|_{L^2} \leq C(1+t)^{-\frac{1}{4\beta}}\|U_0\|_{L^2\cap L^1}
\end{align}
where  $U^{\ell}=\mathcal{F}^{-1}(\mathbf{1}_{|\xi|\leq 1}\widehat{U})$.
If $\beta=0$, then
\begin{align}
    \|U^\ell\|_{L^2} \leq Ce^{-t}\|U_0\|_{L^2}.
\end{align}
 \end{enumerate}
\end{theorem}

The proof of Theorem \ref{thm:decayKalman} is carried out in Section \ref{sec:3}. The main idea is the construction of a suitable Lyapunov functional in the spirit of \cite{BZ}. Compared to the formula \eqref{LyaSym}, in high frequencies, the functional reads 
\begin{align}\label{LyaNonSymIntroHF}
\L^h(\xi,t)=\dfrac{1}{2}|\widehat{U}(\xi,t)|^2+\sum_{k=1}^{K}\eps_k|\xi|^{-2k}\textrm{Re}\l B^s(\i \xi A+B^a)^{k-1}\widehat U(\xi,t),B^s(\i \xi A+B^a)^k\widehat U(\xi,t)\r,
\end{align}
and, in low frequencies, 
\begin{align}\label{LyaNonSymIntroBF}
\L^\ell(\xi,t)=\dfrac{1}{2}|\widehat{U}(\xi,t)|^2+\sum_{k=1}^{K}\eps_k\textrm{Re}\l B^s(\i \xi A+B^a)^{k-1}\widehat U(\xi,t),B^s(\i \xi A+B^a)^k\widehat U(\xi,t)\r.
\end{align}
Taking their time-derivative and choosing the $\varepsilon_k$ small enough, one gets
\[
\frac{d}{dt}\L^h + c_1 \sum_{k=0}^{K}|\xi|^{-2k}|B^s(\i \xi A+B^a)^k \widehat{U}|^2 \leq 0,
\]
and
\[
\frac{d}{dt}\L^\ell + c_2 \sum_{k=0}^{K}|B^s(\i \xi A+B^a)^k \widehat{U}|^2 \leq 0,
\]
for constants $c_1,c_2>0$. At which point, thanks to Lemma \ref{lemma:equiv} and to the equivalence of $\L^{h}$ and $\L^\ell$ to the energy $|\widehat{U}(\xi,t)|^2$, one concludes by Gr\"onwall's arguments.

\begin{remark}
Theorem \ref{thm:decayKalman} is a direct extension of the approach developed in \cite{BZ} to the non-symmetric relaxation case. When $B^a=0$, one has $\alpha=0$ and $\beta=1$ leading to the decay rates obtained in the symmetric case \eqref{decay} and \eqref{decayhf}. 
\end{remark}
\begin{remark}\label{rq:nonlin}
Our approach can be employed to derive a priori estimates for nonlinear systems, such as the Timoshenko system with or without memory (see Section \ref{sec:applications}), and justify global-in-time well-posedness and time-asymptotic results for initial data close to constant equilibrium.  Indeed, for nonlinear systems whose linearization around equilibrium satisfies the inhomogeneous Kalman rank condition, such results can be established using classical energy methods and a bootstrap argument.
 A particular instance of this can be found in \cite{CBPSX} where techniques similar to the ones developed in the present paper are applied to deal with the nonlinear Euler-Maxwell system, whose high-frequency loss of regularity exponent is $\alpha=1$.
 \end{remark}

\begin{example}[Decay estimates for the linear Sugimoto model \eqref{sugimoto_lin}] 
\label{ex:sugi}
Theorem \ref{thm:decayKalman} can be directly applied to the Sugimoto model \eqref{sugimoto_lin}. First, we verify the inhomogeneous Kalman condition \eqref{condition:span}. We have
\[
B^s(i\xi A + B^a) = 
\left[\begin{array}{ccc}
    0 & 0 &0 \\
    -\Omega &0 &\omega \\
    0 &0 &0
\end{array} \right] \quad \text{ and } \quad 
B^s(i\xi A + B^a)^2 = \left[\begin{array}{ccc}
    0 & 0 &0 \\
    -a \eps \Omega \i \xi  & -\eps (\Omega^2 + \omega^2) &0 \\
    0 &0 &0
\end{array} \right].
\]
It is then straightforward to observe that \eqref{condition:span} fails for $K=1$ but holds for $K=2$. Consequently, Theorem \ref{thm:decayKalman} entails that the solution decays asymptotically, with a rate given by Lemma \ref{lemma:equiv}. In high frequencies, there exists a $C>0$ such that
\begin{align*}
    \sum_{k=0}^2 |\xi|^{-2k}|B^s(i\xi A + B^a)^k\hat{U}|^2 &= |\hat{\varphi}|^2 + \frac{1}{|\xi|^2}|-\Omega \hat{u} + \omega \hat{p}|^2 + \frac{1}{|\xi|^4}|-a \eps \Omega i \xi \hat{u} - \eps(\Omega^2 + \omega^2) \hat{\varphi}|^2 \\ &\geq C|\hat{\varphi}|^2 + \frac{C}{|\xi|^2}|\hat{u}|^2 + \frac{C}{|\xi|^2}|\hat{p}|^2 \\
    &\geq \frac{C}{|\xi|^2}|\hat{U}|^2,
\end{align*}
whereas in low frequencies, 
\[
\sum_{k=0}^2 |B^s(i\xi A + B^a)^k\hat{U}|^2 = |\hat{\varphi}|^2 + |-\Omega \hat{u} + \omega \hat{p}|^2 + |-a \eps \Omega i \xi \hat{u} - \eps(\Omega^2 + \omega^2) \hat{\varphi}|^2 \geq C|\xi|^2|\hat{U}|^2.
\]
With the notation of Theorem \ref{thm:decayKalman}, we thus have $\alpha=1$ and $\beta=1$, implying
\begin{align*}
    \|(u,\varphi,p)^h\|_{L^2} \leq C(1+t)^{-{\frac{\gamma}{2}}}\|(u_0,\varphi_0,p_0)\|_{H^{\gamma}}
\end{align*}
and
\begin{align*}
    \|(u,\varphi,p)^\ell\|_{L^2} \leq C(1+t)^{-\frac{1}{4}}\|(u_0,\varphi_0,p_0)\|_{L^2\cap L^1},
\end{align*}
which corresponds to the decay estimates expected from the spectral analysis of the model. Combined with standard commutator estimates to deal with the nonlinearity, the computations in the proof of Theorem \ref{thm:decayKalman} provide a way to justify a global well-posedness result for the Sugimoto model for initial data $(u_0,\varphi_0,p_0)$ being sufficiently small in $L^1\cap H^{s}$ for $s>d/2+3$. Indeed, such a regularity assumption on the initial data ensures the $L^1$-in-time integrability of the Lipschitz norm of the solution.
\end{example}

Our computations actually lead to sharper decay estimates in Besov-type norms.
\begin{theorem}
    Let $K\geq 1$ and assume that $(B^s,A,B^a)$ satisfies the inhomogeneous Kalman rank condition of order $K$. Let $\alpha$ and $\beta$ be the smallest integers such that \eqref{decay:alpha} and \eqref{decay:beta} hold, and assume that the initial data satisfies $U_0\in \dot{B}^{-\sigma}_{2,\infty} \cap \dot{B}_{2,\infty}^{\gamma\alpha}$ for $\sigma>0$ and $\gamma>0$. Then, there exists a constant $C>0$ such that
    \begin{enumerate}
        \item In high frequencies: if $\alpha\geq 1$, the solution of \eqref{eq:main} satisfies
    \begin{align}
    \|U^h\|_{L^2} \leq C(1+t)^{-\frac{\gamma}{2}}\|U_0\|_{L^2\cap \dot{B}_{2,\infty}^{\gamma\alpha}}.
\end{align}
\item In low frequencies: if $\beta\geq 1$, the solution of \eqref{eq:main} satisfies
\begin{align}
    \|U^\ell\|_{L^2} \leq C(1+t)^{-\frac{\sigma}{2\beta}}\|U_0\|_{\dot{B}^{-\sigma}_{2,\infty}}.
\end{align}
 \end{enumerate}
 As in Theorem \ref{thm:decayKalman}, for $\alpha=0$ and $\beta=0$, exponential decay is recovered in the corresponding regimes.
\end{theorem}

\subsection{Limitations of Kalman-based analysis}
\label{sub:identify}
The main limitation of Theorem \ref{thm:decayKalman} resides in the fact that it is not easy to estimate the coefficients $\alpha$ and $\beta$ appearing in Lemma \ref{lemma:equiv}. Indeed, while for simple systems it may be possible to deduce the optimal value for $\alpha$ and $\beta$, if the system becomes too large, the computation quickly becomes prohibitive. Moreover, the issue can be even more subtle, as the following example shows.

\subsubsection{A system with cancellation for specific coefficients} Consider the toy-model
\begin{equation} \label{2by2}
\begin{dcases}
    \partial_t u +a\partial_x u -v + u = 0, \\
    \partial_t v +b\partial_xv+u = 0.
\end{dcases}  
\end{equation}
This system certainly falls into the non-symmetric class, as
\[ A = 
\left[\begin{array}{cc}
    a & 0  \\
    0 & b 
\end{array} \right], \quad B^a = 
\left[\begin{array}{cc}
     0 &1 \\
     -1 &0
\end{array} \right]\quad \text{and }\quad
B^s = 
\left[\begin{array}{cc}
    1 & 0  \\
    0 & 0
\end{array} \right].
\]
\noindent Furthermore, one can easily check the basic energy equality
\[
\frac{1}{2}\frac{d}{dt}(|\widehat{u}|^2+ |\widehat{v}|^2) + |\widehat{u}|^2 = 0,
\]
showing partial dissipation. Let us work in high frequencies for the time being. In order to recover dissipation for the $v$ variable, one could simply multiply the first equation by $-\widehat{v}$ in Fourier space, obtaining
\begin{equation*}
\begin{aligned}
&-\frac{d}{dt}{\rm Re}\,\l \widehat{u},\widehat{v}\r+ |\widehat{v}|^2-|\widehat{u}|^2+{\rm Re}\l \widehat{u},\widehat{v}\r ={\rm Re}\,(a \l i\xi \widehat{u},\widehat{v}\r+b\l \widehat{u},i\xi \widehat{v}\r).
\end{aligned}
\end{equation*}
At this point, an interesting situation arises. If $a \neq b$, then the right-hand side is nonzero. Therefore, in order to control it we need to multiply both sides by $|\xi|^{-2}$, recovering the control of $|\xi|^{-2}|\widehat{v}|^2$ and thus effectively losing a derivative. However, if $a = b$, the right-hand side vanishes and no loss of a derivative occurs. The final rate of decay using classical Lyapunov theory reads
\begin{equation*}
\begin{aligned}
& \Vert (u,v,w)^h(t) \Vert_{L^{2}} \le C e^{-ct}\Vert (u_0,v_0,w_0)\Vert_{L^2}, &&\quad \text{ if } a = b
\end{aligned}
\end{equation*}
and
\begin{equation*}
\begin{aligned}
&\Vert  (u,v,w)^h(t) \Vert_{L^2} \le C t^{-\frac{\gamma}{2}} \Vert( u_0,v_0,w_0)\Vert_{H^{\gamma}},  &&\quad \text{ if } a \neq b.
\end{aligned}
\end{equation*}
Capturing such behaviour using our Kalman analysis does not appear to be possible even for such a simple system. Indeed, Lemma \ref{lemma:equiv} leads to the exponent $\alpha=1$ as
\[
|B^s\widehat{U}|^2 + \frac{1}{|\xi|^2}|B^s(iA\xi + B^a)\widehat{U}|^2 = |\widehat{u}|^2 + \frac{1}{|\xi|^2}|i a \xi \widehat{u} - v|^2 \geq \frac{c}{|\xi|^2}(|\widehat{u}|^2 + |\widehat{v}|^2).
\] 
Since the parameter $b$ does not appear in the above equation nor in the Kalman matrix, the norm recovered by our Kalman analysis \emph{cannot} capture the cancellation. Hence the inhomogeneous Kalman condition cannot predict the sharp decay rate.

\begin{remark} In this case, the improvement in the decay rates actually comes from the fact that if $a=b$ in \eqref{2by2} then by a Galilean change of frame $(x,t)\rightarrow (x-at,t)$, the system reduces to an ODE for which exponential decay holds. Nevertheless, the above computations still show the lack of efficiency of the inhomogeneous Kalman condition to observe such cancellation. In Section \ref{sec:applications}, we provide additional examples for which there does not exist a transformation of the space-time frame (to the best of our knowledge) capable of simplifying the system.
\end{remark}

\subsection{Improving the Kalman approach}
The issues encountered by our Kalman analysis in estimating $\alpha$ and $\beta$ can be seen as an unfortunate -- yet natural -- consequence of what makes the technique so general. Indeed, it is due to the complexity of the Lyapunov functional \eqref{LyaNonSymIntroHF}-\eqref{LyaNonSymIntroBF}, which collects, for every $k=1,\ldots,N-1$, all the Kalman-type interaction between the inhomogeneous operators $B^s(i\xi A+B^a)^k$ and $B^s(i\xi A+B^a)^{k+1}$. In practice, most of these interactions are useless, in the sense that they do not provide additional decay information, and only render the computation of $\alpha$ and $\beta$ more tedious. In that regard, our aim is to simplify the Lyapunov functionals by only keeping the useful terms.
 

The second part of the paper, starting from Section \ref{sec:4}, is dedicated to tackling this issue by developing a form of \textit{inhomogeneous hypocoercivity}. Specifically, we design an algorithm to construct the Lyapunov functionals so that, in the end, a sharp estimate for the decay rate of the solution can be easily computed. With respect to \eqref{LyaNonSymIntroHF}-\eqref{LyaNonSymIntroBF}, these new functionals are much simpler and contain only a few terms. Furthermore, as a byproduct of our algorithm, we provide explicit, algebraic conditions on $A$, $B^a$ and $B^s$ which, if satisfied, produce a cancellation which improves the final decay rate, in the spirit of the system \eqref{2by2}. Importantly, we observe that to capture such cancellations in the estimates, interactions that are not of Kalman-type must be used.

The main idea behind the algorithm is to select, at every step, a new term to add to the functional, which recovers the dissipation for a new variable (or combination of variables). This is done by charting a path on a binary tree, whose nodes are suitable products of $B^s, A$ and $B^a$ (see Figure \ref{fig:tree}). For every node, a rank condition is checked to decide whether adding the term associated with the current node will recover dissipation for a new variable in the final computations. Once sufficient nodes have been selected, there is no more dissipation to be gained, and the algorithm stops. The final functional built this way is guaranteed to decay with a quantifiable rate and, being equivalent to the energy of the system, entails the result.

\begin{figure}[!ht]
    \begin{tikzpicture}[scale=1]
    [level distance=55mm,
        every node/.style={fill=green!20,inner sep=1pt},
        level 1/.style={sibling distance=40mm,nodes={fill=red!20}},
        level 2/.style={sibling distance=45mm,nodes={fill=red!20}}]
    \node[fill=green!20,circle] {$B^s$}
        child[sibling distance=40mm]{node[fill=green!20]{$B^sA$}
            child[sibling distance=20mm]{node[fill=green!20]{$B^sA^2$}}
            child[sibling distance=20mm]{node[fill=green!20]{$B^sAB^a$}}}
        child[sibling distance=40mm]{node[fill=green!20]{$B^sB^a$}
            child[sibling distance=20mm]{node[fill=green!20]{$B^sB^aA$}}
            child[sibling distance=20mm]{node[fill=green!20]{$B^s(B^a)^2$}}};
    \end{tikzpicture}
        \caption{The first few nodes of the tree $\mathsf{T}$.}
        \label{fig:tree}
    \end{figure}
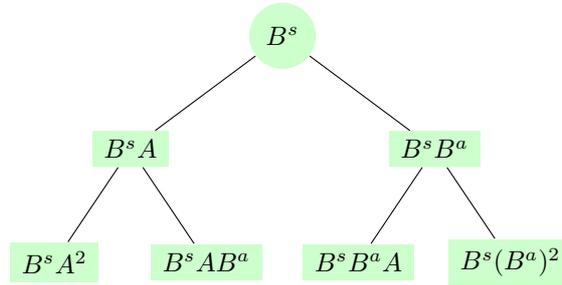

\section{Proof of Theorem \ref{thm:decayKalman}}
\label{sec:3}

\noindent
To analyze the decay of the solution, we employ the Littlewood-Paley decomposition which is convenient for dividing the frequency space into low and high frequencies. Moreover, it allows us to derive sharp decay estimates in Besov spaces. Note that the computations below can be performed in a standard frequency-pointwise manner (with more cumbersome notations); the corresponding computations are described in Remarks \ref{noLPHF} and \ref{noLPHF2}.

\subsection{Littlewood-Paley decomposition}

 Throughout the paper, we fix a homogeneous Littlewood-Paley localisation operator $(\ddj)_{j\in\Z}$
that is defined  by 
$$\ddj\triangleq\mathcal{F}^{-1}(\varphi(2^{-j}\cdot )\mathcal{F}u),\quad\varphi(\xi)\triangleq \chi(\xi/2)-\chi(\xi)$$
where $\chi=\chi(\xi)$ stands for a smooth radial non-increasing function with range in $[0,1],$ supported in  $]-4/3,4/3[$ and
such that $\chi\equiv1$ on $[-3/4,3/4]$.  This definition ensures that, for any tempered distribution $f\in \mathcal{S}'(\R)$, the support of the Fourier transform of $\ddj f$ is localized in an annulus, i.e. there exist two constants $c>0$ and $C>0$ such that
\begin{align}
    \textrm{supp}\left(\widehat{\ddj f}\right)\subset \{\frac{3}{4}2^j\leq |\xi|\leq {\frac{8}{3}}2^{j}\}.
\end{align}
The Littlewood–Paley decomposition of a general tempered distribution
$$
f=\sum_{j\in\mathbb{Z}}\dot{\Delta}_j f.
$$
The above equality holds only in the subset $S_h'(\R)$ of $S'(\R)$ modulo polynomials. To simplify the notation, for any $f\in \mathcal{S}'(\R)$ and $j\in \Z$, we note
\begin{align}
    \ddj f:=f_j.
\end{align}
In the paper, we will use repeatedly the Bernstein property stating that for such localized distribution, the differentiation operator acts as a homothety. 
\begin{lemma}[Bernstein properties]\label{lem:Bernstein} For $f\in \mathcal{S}_h'(\R)$ and $j\in\mathbb{Z}$, we have
\begin{align}
  \|\partial_xf_j\|_{L^2} \sim  2^{j}\|f_j\|_{L^2}.
\end{align}
Moreover, for $f=f^h+f^\ell$ with $f^h=\mathcal{F}^{-1}(\mathbf{1}_{|\xi|\geq 1}\mathcal{F}f)$ and $f^\ell=\mathcal{F}^{-1}(\mathbf{1}_{|\xi|\geq 1}\mathcal{F}f)$, it holds that
 \begin{align}
    \|f^h_j\|_{L^2} \leq C  2^{j}\|f^h_j\|_{L^2}\quad\text{and}\quad \|f^\ell_j\|_{L^2} \leq C 2^{-j}\|f^\ell_j\|_{L^2}
\end{align}
for some uniform constant $C>0$.
\end{lemma}
As a straightforward consequence of the Bernstein property and Lemma \ref{lemma:equiv}, we obtain the following.
\begin{lemma}
There exists a uniform constant $c>0$ such that, for $j\geq 0$,
\begin{align}\label{decay:alphaLP}
\sum_{k=0}^{K}2^{-2jk}\|B^s(A\partial_x+B^a)^kU_j\|_{L^2}^2 \geq c 2^{-2j\alpha}\|U_j\|_{L^2}^2,
\end{align}
 and, for $j\leq0$,
\begin{align}\label{decay:betaLP}
\sum_{k=0}^{K}\|B^s(A\partial_x+B^a)^kU_j\|_{L^2}^2 \geq c 2^{2j\beta}\|U_j\|_{L^2}^2,
\end{align}
where $\alpha$ and $\beta$ are given by Lemma \ref{lemma:equiv}.
\end{lemma}

We are now ready to analyze the decay properties of \eqref{eq:main}. We focus on the high-frequency regime first.

\subsection{High-frequency decay estimates}

We have the following proposition.


\begin{proposition}[High-frequency decay]
\label{prophf}
Let $U^h=\mathcal{F}^{-1}(\mathbf{1}_{|\xi|\geq 1}\mathcal{F}U)$ and $j\in\mathbb{Z}$. There exists a high-frequency Lyapunov functional $\L^h_j\sim \|U_j^{h}\|_{L^2}^2$ such that for all $t>0$,
\begin{align}
    \dfrac{d}{dt}\L^h_j+c 2^{-2j\alpha}\L^h_j\leq 0.\label{LLY}
\end{align}
Consequently, for any $\gamma>0$ and $t>0$, we have
\begin{align}
    \|U^h\|_{L^2} \leq C(1+t)^{-\frac{\gamma}{2}}\|U_0\|_{H^{\gamma\alpha}}.\label{section2:high}
    \end{align}
    Here, $\alpha$ is given by Lemma \ref{lemma:equiv}, and $c, C>0$ are universal constants. 
\end{proposition}
\begin{proof}
Let $(\eps_k)_{1\leq k\leq K}$ be a sequence of positive real numbers to be chosen later. We define the high-frequency perturbed energy functional
\begin{align}
\label{eq:lyaph}
\L^h_j=\dfrac{1}{2}\|U_j^{h}\|_{L^2}^2+\sum_{k=1}^{K}\eps_k2^{-2jk}\bigl(B^s(A\partial_x+B^a)^{k-1}U^{h}_j,B^s(A\partial_x+B^a)^kU^{h}_j\bigl).
\end{align}
Here and in what follows, $\bigl( \cdot,\cdot \bigr)$ denotes the inner
product in $L^2$. Differentiating in time $\L^h_j$ and employing \eqref{partialdiss}, we obtain
\begin{align}
    \dfrac{d}{dt}\L^h_j+\kappa\|U^{h}_j\|_{L^2}^2+\sum_{k=1}^{K}\eps_k2^{-2jk}\|B^s(A\partial_x+B^a)^kU^{h}_j\|_{L^2}^2&\leq  I^j_1+I^j_2+I^j_3,
\end{align}
where 
\begin{equation*}
    \begin{dcases}        I^1_{j}=\sum_{k=1}^{K}\eps_k2^{-2jk}\bigl(B^s(A\partial_x+B^a)^{k-1}B^sU^{h}_j,B^s(A\partial_x+B^a)^kU^{h}_j\bigl),
\\
I^2_j=\sum_{k=1}^{K}\eps_k2^{-2jk}\bigl(B^s(A\partial_x+B^a)^{k-1}U^{h}_j,B^s(A\partial_x+B^a)^kB^sU^{h}_j\bigl),
\\
I^3_j=\sum_{k=1}^{K}\eps_k2^{-2jk}\bigl(B^s(A\partial_x+B^a)^{k-1}U^{h}_j,B^s(A\partial_x+B^a)^{k+1}U^{h}_j\bigl).
    \end{dcases}
\end{equation*}
Controlling $I^1_{j}$, $I^2_j$ and $I^3_j$ follows closely from \cite{BZ} and \cite{CBD2}. We adapt it to the inhomogeneous hypocoercivity framework as follows.
\medbreak $\bullet$ The terms $I^1_{k,j}:= 2^{-2jk}\eps_k\bigl(B^s(A\partial_x+B^a)^{k-1}B^s  U^{h}_j, B(A\partial_x+B^a)^k U^{h}_j\bigr)$
with $k\in\{1,\cdots, K\}$: Due to Bernstein properties in Lemma \ref{lem:Bernstein} and the fact that the matrices $A$, $B^s$ are bounded, we obtain
\begin{equation*}
\begin{aligned}
 |I^1_{k,j}| &\leq C\eps_k2^{-2jk}\|B^s\partial_x^{k-1}U^h_{j}\|_{L^2} \|B(A\partial_x+B^a)^k U^{h}_j\|_{L^2} 
\\&\leq \frac{\kappa}{4K}{\| B^sU^h_{j}\|_{L^2}^2}+\frac{C \eps_k^2}{\kappa}2^{-2jk}\| B(A\partial_x+B^a)^k  U^{h}_j\|_{L^2}^2.
\end{aligned}
\end{equation*}
\newline  $\bullet$  The term $I^2_{1,j}:=2^{-2j} \eps_1\bigl(B^sU^{h}_j, B^s(A\partial_x+B^a) B^sU^{h}_j\bigr)_{L^2}$: One has
\begin{equation}\nonumber
|I^2_{1,j}| \leq C2^{-2j}\eps_1 \|B^sU^{h}_j\|_{L^2}\| B^s(A\partial_x+B^a)B^sU^{h}_j\|_{L^2}\\
\leq \frac{\kappa}{4K} \|B^sU^{h}_j\|_{L^2}^2+ \frac{C\eps_1^2}{\eps_{0}} \|B^s\partial_xU^{h}_j\|_{L^2}^2.
\end{equation}
\newline  $\bullet$ The terms  $I^2_{k,j}:=\eps_k2^{-2jk}\bigl(B^s(A\partial_x+B^a)^{k-1}  U^{h}_j, B^s(A\partial_x+B^a)^k B^s \partial_xU^{h}_j\bigr)$ with  $k\in\{2,\cdots, K\}$ (if $n\geq3$): After integrating by parts, we have
$$\begin{aligned}
 |I^2_{k,j}|&=2^{-2jk}\eps_k |\bigl(B^s(A\partial_x+B^a)^{k-1}  \partial_x U^{h}_j, B^s(A\partial_x+B^a)^k B^s U^{h}_j\bigr)_{L^2} |\\&\leq C2^{-2jk}\eps_k \|B^s(A\partial_x+B^a)^{k-1} \partial_x U^{h}_j\|_{L^2}\|B^s \partial_x^kU^{h}_j\|_{L^2}\\
&\leq \frac{\kappa}{4K}{\| B^s U^{h}_j\|_{L^2}^2}+\frac{C\eps_{k-1}^2}{\kappa}2^{-2jk}\| B^s(A\partial_x+B^a)^{k-1} U^{h}_j\|_{L^2}^2.
\end{aligned}$$
\newline  $\bullet$ The terms $I^3_{j,k}:= \eps_k 2^{-2jk} \bigl( B^s(A\partial_x+B^a)^{k-1}  U^{h}_j,  B^s(A\partial_x+B^a)^{k+1}  U^{h}_j\bigr)_{L^2}$ with $k\in\{1,\cdots, K-1\}$ (if $n\geq3$): A similar argument yields
\begin{align*}
 |I^3_{k,j}|&=\eps_k 2^{-2jk}| \bigl( B^s(A\partial_x+B^a)^{k-1} U^{h}_j,  B^s(A\partial_x+B^a)^{k+1}  U^{h}_j\bigr)_{L^2}|
\\&\leq \frac{\eps_{k-1}}{8}  2^{-2j(k-1)}\|B^s(A\partial_x+B^a)^{k-1}  U^{h}_j\|_{L^2}^2 +\frac{C\eps_k^2}{\eps_{k-1}}2^{-2j(k+1)}  \| B^s(A\partial_x+B^a)^{k+1} U^{h}_j\|_{L^2}^2. 
\end{align*}
\newline  $\bullet$  The term $I^3_{K}:= 2^{-2jK}\eps_{K} \bigl( B^s(A\partial_x+B^a)^{K-1}  U^{h}_j,  B^s(A\partial_x+B^a)^{K+1} U^{h}_j\bigr)_{L^2}$:
Owing to the inhomogeneous Kalman rank condition and the Cayley–Hamilton theorem, there exist coefficients
$c_{*}^q$ ($q=0,1,...,K$) such that
\begin{align}
    &B^s(A\partial_x+B^a)^{K+1} = \sum_{q=0}^{K}  2^jc_{*}^j B^s(A\partial_x+B^a)^q.\label{CH}
\end{align}
Consequently, one gets
$$\begin{aligned}
|I^3_{K}|&\leq \eps_{K}2^{-2jK}\| B^s(A\partial_x+B^a)^{K-1}  U^{h}_j\|_{L^2}\sum_{q=0}^{K}2^{j}c_{*}^j 
\|B^s(A\partial_x+B^a)^q  U^{h}_j\|_{L^2} \\
&\leq  \frac{\eps_{K-1}}{8}2^{-2j(K-1)}\| B^s(A\partial_x+B^a)^{K-1}  U^{h}_j\|_{L^2}^2+\frac{C\eps_{K}^2}{\eps_{K-1}}2^{-2jK}\sum_{q=0}^{K} \| B^s(A\partial_x+B^a)^q  U^{h}_j\|_{L^2}^2
\end{aligned} $$
According to \cite{BZ}, we set $\varepsilon_k=\varepsilon^{m_k}$ with  $\varepsilon$ small enough and 
$m_1,\cdots,m_{K}$ satisfying for some $\delta>0$ (that can be taken arbitrarily small): 
$$m_k\geq \frac{m_{k-1}+m_{k+1}}2+\delta \quad \text{and} \quad
m_{K}\geq\frac{m_{k}+m_{K-1}}2+ \delta, \quad k=1,\cdots,K-1.$$
Gathering the above estimates, we deduce that there exists a constant $\eta>0$ such that
\begin{equation}
\label{eq:energyhf}
    \dfrac{d}{dt}\L^h_j  +\eta \sum_{k=0}^{K}2^{-2jk}\|B^s(A\partial_x+B^a)^kU^{h}_j\|_{L^2}^2 \leq 0.
\end{equation}
Moreover, for a suitably small sequence $\varepsilon$, we have 
$$
\L^h_j \sim \|U^{h}_{j}\|_{L^2}^2,
$$
and due to the property \eqref{decay:alpha} related to the inhomogeneous Kalman rank condition,
$$
\sum_{k=0}^{K}2^{-2jk}\|B^s(A\partial_x+B^a)^kU^{h}_j\|_{L^2}^2 \geq c2^{-2j\alpha}\|U^{h}_{j}\|_{L^2}^2.
$$
Therefore, 
\begin{align}
    \dfrac{d}{dt}\L^h_j + c 2^{-2j\alpha}\L^h_j\leq 0.\label{Lyapun:high}
\end{align}
 From this and Gr\"onwall's lemma, we deduce that
\begin{equation*}
\begin{aligned}
&\L^h_j(t)\lesssim e^{-c 2^{-2j\alpha} t}\L^h_j(0).
\end{aligned}
\end{equation*}
This leads to
\begin{equation}\label{high1}
\begin{aligned}
&\|U^{h}\|_{L^2}^2\sim \sum_{j\in\mathbb{Z}} \|\dot{\Delta}_j U^{h}\|_{L^2}^2\lesssim \sum_{j\in\mathbb{Z}} e^{-c 2^{-2j\alpha} t}\|\dot{\Delta}_{j}U_0^h\|_{L^2}^2\lesssim t^{-\gamma}\|U_0\|_{\dot{H}^{\gamma \alpha}}^2,
\end{aligned}
\end{equation}
where we used the fact that thanks to the transformation $j\rightarrow j'\in\mathbb{Z}$ such that $2^{-2j'\alpha}\leq 2^{-2j\alpha}t\leq 2^{-2j'\alpha+1}$, we have
\begin{equation*}
\begin{aligned}
\sum_{j\in\mathbb{Z}} 2^{-2\gamma \alpha j}e^{-c 2^{-2j\alpha} t} \leq 2 t^{-\gamma}\sum_{j'\in\mathbb{Z}} 2^{-2\gamma \alpha j'}e^{-c 2^{-2j'\alpha}}\lesssim t^{-\gamma}.
\end{aligned}
\end{equation*}
On the other hand, integrating \eqref{Lyapun:high} in time gives directly 
\begin{equation}\label{high2}
\begin{aligned}
&\sup_{t\geq0}\|U^{h}\|_{L^2}^2\lesssim \|U_0\|_{L^2}^2.
\end{aligned}
\end{equation}
Combining \eqref{high1}-\eqref{high2}, we get \eqref{section2:high}.

\end{proof}

\begin{remark}\label{noLPHF}
    Above we employed the Littlewood-Paley decomposition to simplify the proof. It is also possible to justify a similar result with frequency-pointwise functionals. For $|\xi|\geq1$, consider
\begin{align*}
\L^h(\xi,t)=\dfrac{1}{2}|U(\xi,t)|^2+\sum_{k=1}^{K}\eps_k|\xi|^{-2k}\textrm{Re}\l B^s(\i \xi A+B^a)^{k-1}\widehat U(\xi,t),B^s(\i \xi A+B^a)^k\widehat U(\xi,t)\r.
\end{align*}
Following the computations done above, one obtains
\begin{align}\label{noLP}
\dfrac{d}{dt}\L^h+c |\xi|^{-2\alpha}\L^h\leq 0.   
\end{align}
Then, since $\L^h(\xi,t)\sim |\widehat U(\xi,t)|^2$ for $|\xi|\geq1$, solving \eqref{noLP} we deduce the desired decay estimates.
\end{remark}

\subsection{Low-frequency decay estimates}
We have the following proposition.
\begin{proposition}[Low-frequency decay] 
Let $U^{\ell}=\mathcal{F}( \mathbf{I}_{|\xi|\leq 1}\mathcal{F}U )$ and $j\in\mathbb{Z}$. There exists a low-frequency Lyapunov functional $\L^\ell_j\sim \|U^{\ell}_j\|_{L^2}^2$ such that for all time $t>0$,
\begin{align}
    \dfrac{d}{dt}\L^\ell_j+c 2^{2\beta j}\L^\ell_j\leq 0.
\end{align}
Consequently, for any $t>0$, we have
\begin{align}
    \|U^\ell\|_{L^2} \leq C (1+t)^{-\frac{1}{4\beta}}\|U_0\|_{ L^1\cap L^2}.\label{section2:low}
\end{align}
Here, $\beta$ is given by Lemma \ref{lemma:equiv}, and $C, c>0$ are universal constants.
\end{proposition}
\begin{proof}

 We define the low-frequency perturbed energy functional
\begin{align}
\label{eq:lyapl}
\L^\ell_j=\dfrac{1}{2}\|U^{\ell}_j\|_{L^2}^2+\sum_{k=1}^{K}\eps_k\bigl(B^s(A\partial_x+B^a)^{k-1}U^{\ell}_j,B^s(A\partial_x+B^a)^kU^{\ell}_j\bigl).
\end{align}
Differentiating in time $\L^\ell_j$, we obtain
\begin{align}
    \dfrac{d}{dt}\L^\ell_j+\kappa\|U^{\ell}_j\|_{L^2}^2+\sum_{k=1}^{K}\eps_k\|B^s(A\partial_x+B^a)^kU^{\ell}_j\|_{L^2}^2&\leq  J^1_j+J^2_j+J^3_j,
\end{align}
where 
\begin{equation}
    \begin{dcases}        J^1_j=\sum_{k=1}^{K}\eps_k\bigl(B^s(A\partial_x+B^a)^{k-1}B^sU^{\ell}_j,B^s(A\partial_x+B^a)^kU^{\ell}_j\bigl),
\\
J^2_j=\sum_{k=1}^{K}\eps_k\bigl(B^s(A\partial_x+B^a)^{k-1}U^{\ell}_j,B^s(A\partial_x+B^a)^kB^sU^{\ell}_j\bigl),
\\
J^3_j=\sum_{k=1}^{K}\eps_k\bigl(B^s(A\partial_x+B^a)^{k-1}U^{\ell}_j,B^s(A\partial_x+B^a)^{k+1}U^{\ell}_j\bigl).
    \end{dcases}
\end{equation}
Controlling $J^1_j$, $J^2_j$ and $J^3_j$ follows closely from the computations done in the high-frequency section.
In particular, there exists a constant $\mu>0$ such that, choosing the sequence $\{\eps_k\}_{1\leq k\leq K}$ as in the high-frequency section, we have
\begin{align}
    \dfrac{d}{dt}\L^\ell +\mu \sum_{k=0}^{K}\|B^s(A\partial_x+B^a)^kU^{\ell}_j\|_{L^2}^2 \leq 0.
\end{align}
Moreover, for a suitably small $\varepsilon$, using the inhomogeneous Kalman rank condition, we have $$\L^\ell_j \sim \|U_{j}^{\ell}\|_{L^2}^2.$$
Therefore, 
\begin{align}
    \dfrac{d}{dt}\L_j^\ell + 2^{2\beta j}\L_j^\ell\leq 0.\label{Lyapunov:low}
\end{align}
The use of Gr\"onwall's lemma to \eqref{Lyapunov:low} yields
\begin{equation}\nonumber
\begin{aligned}
&\L_j^\ell(t)\lesssim  e^{-2^{2\beta j} t} \L_j^\ell(0),
\end{aligned}
\end{equation}
from which we infer
\begin{equation}\label{low1}
\begin{aligned}
\|U^{\ell}\|_{L^2}^2\sim \sum_{j\in\mathbb{Z}} \|\dot{\Delta}_j U^{\ell}\|_{L^2}^2\lesssim \sum_{j\in\mathbb{Z}} e^{-2^{2\beta j} t}  \|\dot{\Delta}_{j}U_0^{\ell}\|_{L^2}\lesssim t^{-\frac{1}{4\beta}}\|U_0\|_{L^1},
\end{aligned}
\end{equation}
due to the facts that $\|\dot{\Delta}_{j}U_0^{\ell}\|_{L^2}\lesssim 2^{\frac{j}{2}} \|U_0\|_{\dot{B}^{-\frac{1}{2}}_{2,\infty}}$, the embedding $L^1\hookrightarrow \dot{B}^{-\frac{1}{2}}_{2,\infty}$ and
$$
\sum_{j\in\mathbb{Z}} t^{\frac{1}{4\beta}}2^{-\frac{j}{2}}e^{-2^{2\beta j} t}\lesssim  1.
$$
On the other hand, taking advantage of \eqref{Lyapunov:low} and the low-frequency cut-off property, we also have
\begin{equation}\label{low2}
\begin{aligned}
\|U^{\ell}\|_{L^2}\lesssim \|U_0^{\ell}\|_{L^2}\lesssim \|U_0^{\ell}\|_{\dot{B}^{-\frac{1}{2}}_{2,\infty}}\lesssim \|U_0\|_{L^1}.
\end{aligned}
\end{equation}
By \eqref{low1}-\eqref{low2}, the decay estimate \eqref{section2:low} follows.

\end{proof}

\begin{remark}\label{noLPHF2}
As mentioned in Remark \ref{noLPHF}, it is possible to justify a similar result with a frequency-pointwise functional. Consider, for $|\xi|\leq1$,
\begin{align*}
\L^\ell(\xi,t)=\dfrac{1}{2}|U(\xi,t)|^2+\sum_{k=1}^{K}\eps_k\textrm{Re}\bigl(B^s(\i \xi A+B^a)^{k-1}\widehat U(\xi,t),B^s(\i \xi A+B^a)^k\widehat U(\xi,t)\r.
\end{align*}
By similar computations done above, one obtains
\begin{align}\label{noLPLF}
\dfrac{d}{dt}\L^\ell+c|\xi|^{2\beta}\L^\ell\leq 0.   
\end{align}
Then, since for $|\xi|\leq1$ we have $ \L^\ell(\xi,t)\sim |\widehat U(\xi,t)|^2$, solving \eqref{noLPLF} leads to the desired decay estimates.
\end{remark}


\section{An Algorithm to Improve the Lyapunov Functional}
\label{sec:4}
\noindent
While the previous section describes an effective method to prove the decay of the solution for systems with non-symmetric relaxation, it does not offer a straightforward approach for computing the decay rate. Besides, in many physical situations, improvements of the rate of decay are observed for particular choices of coefficients. In this section, we introduce a technique to build new Lyapunov functionals which can capture these improvements and provide an explicit rate of decay.


\begin{assumption}
\label{ass:kalman}
Along this section, we assume that the inhomogeneous Kalman condition of Definition \ref{def:inhomo} holds. 
\end{assumption}

The new functional builds upon the ones described in \eqref{eq:lyaph} and \eqref{eq:lyapl}. The idea is to remove the redundant terms in the expansion of $B^s(A \partial_x + B^a)^k$, specifically  the terms which do not contribute to recovering dissipation in new directions. To this end, we proceed in an algorithmic fashion. We define a binary tree $\mathsf{T}$ as follows (see also Figure \ref{fig:tree}):
\begin{itemize}
    \item The base of the tree is the node $B^s$. We say that the level is $k =0$.
    \item We associate to every node at the level $k $ a matrix $X_k ^i$, where $i$ is the position of the node on that level ($i = 1$ for the leftmost node). At every node $X_k ^i$, the tree splits into two nodes, $X_k ^i A$ (left) and $X_k ^i B^a$ (right). For instance, the first two nodes are $X_1^1 = B^s A$ and $X_1^2 = B^s B^a$, see Figure \ref{fig:tree}.
\end{itemize}
We are also going to need the following definition. 
\begin{definition}
\label{admissible}
    Let $p_1 = 1$ and $q_1 = 1/2$. A sequence $\{p_k , q_k \}_{k  \in \N^*}$ is called admissible if 
    \begin{enumerate}
        \item $p_k , q_k  > 0$ for every $k  \in \N^*$.
        \item For every $k\in \N^*$ it holds 
        \[
        q_k  < p_k  - p_{k -1} <  q_{k -1},
        \]
    \end{enumerate}
\end{definition}
It follows that if $\{p_k ,q_k \}_{k \in \N^*}$ is admissible, then $p_k $ is increasing and $q_k $ is decreasing. Moreover, it is apparent that
\[
p_{k -1} < p_k  - q_k , \quad \text{ and } \quad p_k  + q_k  > p_{k +1}.\]
For the rest of this section, we fix an admissible $\{p_k ,q_k \}_{k \in \N^*}$. We are now in a position to describe the algorithm constructing the Lyapunov functional.

\subsubsection{Initialization} Mimicking the previous section, the first piece of our functional is simply the energy of the system, 
\[
\mathsf{E} := \frac12 \|U_j\|_{L^2}^2.
\]
We have
\[
\frac{d}{dt} \mathsf{E} + \bigl(B^s U_j, U_j\bigl) = 0,
\]
and by the hypotheses on $B^s$,
\[
\frac{d}{dt} \mathsf{E} + \kappa\| U_j\|_{L^2}^2 \leq 0.
\]

\subsubsection{Iteration}
The functional is constructed by charting a suitable path along $\mathsf{T}$. For $k \geq1$, starting from a node $X_{k -1}^i$ at the level $k -1$, we want to define a procedure to decide whether to add the nodes connected to it on the next level or to stop. A priori, there might be different paths for achieving our goal. In order to select a unique path, we make the following choice to scour the tree:
\begin{equation}
\label{alwaysleft}
\tag{L}
    \text{We always begin by considering the \textit{leftmost} node available at the level } k -1.
\end{equation}
Once we have selected a node $X_{k -1}^i$, we consider the matrix
\[
M = \begin{bmatrix}
    B^s, &\ldots, &X^i_{k -1}
\end{bmatrix}^\top,
\]
collecting all of the previously chosen nodes. We call
\[
r := \rank(M).
\]
Furthermore, we define
\[
M_A := \begin{bmatrix}
    B^s, &\ldots, &X^i_{k -1}, &X^i_{k -1}A
\end{bmatrix}^\top,
\]
and
\[
M_{B^a} := \begin{bmatrix}
    B^s, &\ldots, &X^i_{k -1}, &X^i_{k -1}B^a
\end{bmatrix}^\top.
\]
Five possibilities can occur (see Figure \ref{fig1}):
\begin{enumerate}
    \item \textbf{Left:} If 
\begin{equation}
    \label{eq:rankleft}
    \rank(M_A) > r \text{ but } \rank(M_{B^a}) = r,
\end{equation} 
then we add $X_{k -1}^iA$ to the path.
    \item \textbf{Right:} If 
\begin{equation}
    \label{eq:rankright}
    \rank(M_{B^a}) > r \text{ but } \rank(M_A) = r,
\end{equation} 
then we add $X_{k -1}^iB^a$ to the path.
   \item \textbf{Left AND Right:} If $\rank(M_A) > r$ and 
\begin{equation}
\label{eq:rankmix}
\rank\left(\begin{bmatrix}
    B^s, &\ldots, &X_{k -1}^i, &X_{k -1}^iA, &X_{k -1}^iB^a
\end{bmatrix}^\top\right) > \rank(M_A),
\end{equation}
we add both node $X_{k -1}^iA$ and $X_{k -1}^iB^a$ to the path.
\item \textbf{Left OR Right}: If $\rank(M_A) > r$, $\rank(M_{B^a}) > r$ and 
\begin{equation}
\label{eq:rankmix2}
\rank\left(\begin{bmatrix}
    B^s, &\ldots, &X_{k -1}^i, &X_{k -1}^iA, &X_{k -1}^iB^a
\end{bmatrix}^\top\right) = \rank(M_A) = \rank(M_{B^a}),
\end{equation} 
we add node $X_{k -1}^iA$ to the path. In low frequencies, we will instead add $X_{k -1}^iB^a$.
\item \textbf{Stop:} If $\rank(M_A) = \rank(M_{B_a}) = r$, then we stop and move to the next node, namely the first one on the right of $X_{k -1}^i$.  
\end{enumerate}
Whenever a node is added to the path, a certain term will be summed to the final Lyapunov functional. To define the exact term it will be crucial to associate to each chosen node a special number, which we call \textit{discrepancy} $\d_k ^i$. This discrepancy term will be directly related to the decay rates we can recover in the end.

Since the terms are slightly different in high and low frequencies, we split the analysis into two sections. 

\begin{figure}[h!]
    \begin{minipage}[t]{0.22\linewidth}
        \centering
        \begin{tikzpicture}
            [level distance=15mm,
   every node/.style={fill=green!20,inner sep=1pt},
   level 1/.style={sibling distance=10mm,nodes={fill=red!20}},
   level 2/.style={sibling distance=15mm,nodes={fill=red!20}}]
  \node[fill=green!20,circle] {$\ldots$}
     child {node[fill=green!20] {$X_{k-1}^i$}
       child {node[fill=green!20] {$X_{k-1}^iA$}
       }
       child {node[fill=red!20] {$X_{k-1}^iB^a$}
       }
     }
     child {node[fill=red!20,circle] {$\ldots$}
     };
        \end{tikzpicture}
        Left path
    \end{minipage}%
    \begin{minipage}[t]{0.22\linewidth}
        \centering
        \begin{tikzpicture}
         [level distance=15mm,
   every node/.style={fill=green!20,inner sep=1pt},
   level 1/.style={sibling distance=10mm,nodes={fill=red!20}},
   level 2/.style={sibling distance=15mm,nodes={fill=red!20}}]
  \node[fill=green!20,circle] {$\ldots$}
     child {node[fill=green!20] {$X_{k-1}^i$}
       child {node[fill=red!20] {$X_{k-1}^iA$}
       }
       child {node[fill=green!20] {$X_{k-1}^iB^a$}
       }
     }
     child {node[fill=red!20,circle] {$\ldots$}
     };
        \end{tikzpicture}
        Right path
    \end{minipage}
    \begin{minipage}[t]{0.22\linewidth}
        \centering
        \begin{tikzpicture}
            [level distance=15mm,
   every node/.style={fill=green!20,inner sep=1pt},
   level 1/.style={sibling distance=10mm,nodes={fill=red!20}},
   level 2/.style={sibling distance=15mm,nodes={fill=red!20}}]
  \node[fill=green!20,circle] {$\ldots$}
     child {node[fill=green!20] {$X_{k-1}^i$}
       child {node[fill=green!20] {$X_{k-1}^iA$}
       }
       child {node[fill=green!20] {$X_{k-1}^iB^a$}
       }
     }
     child {node[fill=red!20,circle] {$\ldots$}
     };
        \end{tikzpicture}
        Mixed path
    \end{minipage}%
    \begin{minipage}[t]{0.22\linewidth}
        \centering
        \begin{tikzpicture}
            [level distance=15mm,
   every node/.style={fill=green!20,inner sep=1pt},
   level 1/.style={sibling distance=10mm,nodes={fill=red!20}},
   level 2/.style={sibling distance=15mm,nodes={fill=red!20}}]
  \node[fill=green!20,circle] {$\ldots$}
     child {node[fill=green!20] {$X_{k-1}^i$}
       child {node[fill=red!20] {$X_{k-1}^iA$}
       }
       child {node[fill=red!20] {$X_{k-1}^iB^a$}
       }
     }
     child {node[fill=green!20,circle] {$\ldots$}
     };
        \end{tikzpicture}
        Stop
    \end{minipage}%
    \caption{The four possible routes to follow, starting from a node $X_k^i$}
    \label{fig1}
\end{figure}
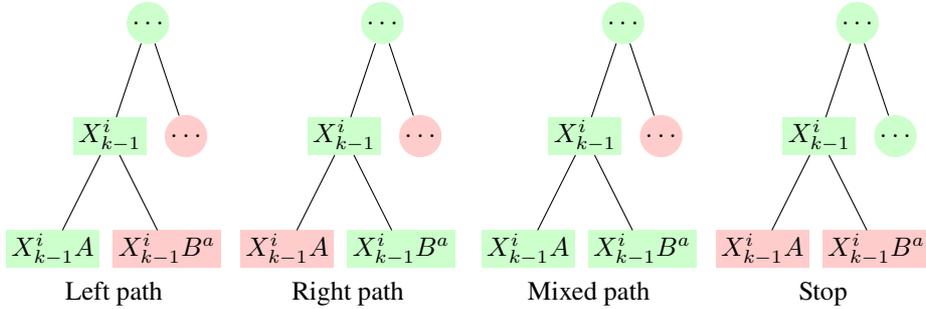

\section{High-frequency Analysis}
\label{sec:ImproveHF}
\noindent
\subsection{Building blocks}
We now proceed to describe in detail the building blocks of the final functional in high frequencies.
In what follows, we will omit the position $i$ of the node within the level $k $ for clarity of exposition. In addition, we always use the high-frequency variable $U^{h}=\mathcal{F}^{-1}(\mathbf{1}_{|\xi|\geq 1}\widehat{U})$ and simplify the notation by omitting the superscript $^h$.

\subsection*{Case (1): Left} In this case, we associate the term
\[
\Phi^k _A := 2^{-2j}\eps^{p_k }\bigl(X_{k -1} U_j ,\, X_{k -1} A\partial_x U_j  \bigl),
\]
and the discrepancy \[\d_k  = 0.\]
We have the following lemma.
\begin{lemma}
\label{lemma:left}
Let $0<\eps<1$.   It holds
    \begin{align}
    \label{leftpath}
    \begin{split}
        \partial_t \Phi^k _A + \frac{\eps^{p_k }}{2}\| X_{k -1} A U_j \|_{L^2}^2 &\leq C \eps^{p_k  - q_k } \|X_{k -1} U_j \|_{L^2}^2 + C\eps^{p_k }\sum_{i = 0}^{k -2} \|X_i U_j \|_{L^2}^2 \\ 
        &+ \eps^{p_k  + q_k } \|X_{k -1}A^2 U_j \|_{L^2}^2 + \eps^{p_k  + q_k }2^{-2j}\|X_{k -1}AB^aU_j \|_{L^2}^2.
    \end{split}
    \end{align}
\end{lemma}
\begin{proof}
A direct computation yields
    \begin{align}
    \label{eq:left}
    \begin{split}
        2^{-2j}\partial_t \bigl(X_{k -1} U_j ,\, X_{k -1} A\partial_x U_j  \bigl) &+2^{-2j}\| X_{k -1} A \partial_x U_j \|_{L^2}^2\\
        =&  
        - 2^{-2j}\bigl(X_{k -1} B^s U_j , \, X_{k -1} A \partial_x U_j \bigl) \\
        &- 2^{-2j}\bigl(X_{k -1} U_j , \, X_{k -1} A B^s\partial_x U_j \bigl) \\
        &- 2^{-2j}\bigl(X_{k -1} B^a U_j , \, X_{k -1} A\partial_x U_j \bigl) \\
        &- 2^{-2j}\bigl(X_{k -1} U_j , \, X_{k -1} A^2 \partial_{xx} U_j \bigl) \\
        &- 2^{-2j}\bigl(X_{k -1} U_j , \, X_{k -1} A B^a\partial_x U_j \bigl).
    \end{split}
\end{align}
The support of $\widehat{U_j}$ ensures that
$$
2^{-2j}\| X_{k -1} A \partial_x U_j \|_{L^2}^2\geq \frac{9}{16}\| X_{k -1} A U_j \|_{L^2}^2.
$$
For the first term, we deduce from Lemma \ref{lem:Bernstein} that
\begin{align*}
2^{-2j}|\bigl(X_{k -1} B^s U_j , \, X_{k -1} A \partial_x U_j \bigl)|&\leq 2^{-2j} \| X_{k -1} B^s U_j \|_{L^2} \|X_{k -1} A \partial_xU_j \|_{L^2} 
\\
&\leq C 2^{-j}\| X_{k -1} B^s U_j \|_{L^2}^2 \|X_{k -1} A U_j \|_{L^2}^2
\\
&\leq C 2^{-2j} \|X_{k -1}\|^2_O \|B^s U_j \|_{L^2}^2 + \frac{1}{32}\|X_{k -1} A U_j \|_{L^2}^2,
\end{align*}
where $\|X_{k -1}\|^2_O$ is the operator norm of $X_{k -1}$.
For the second term, using the Bernstein properties in Lemma \ref{lem:Bernstein} again, we obtain
\[
2^{-2j}|\bigl(X_{k -1} U_j , \, X_{k -1} A B^s\partial_x U_j \bigl)| \leq  C\|X_{k -1}A\|^2_O\|B^sU_j \|_{L^2}^2+\|X_{k -1}U_j \|_{L^2}^2.
\]
Regarding the third term, we recall that $\rank(M_{B_a}) = r$. Hence, the Cayley–Hamilton theorem ensures that
\[
\|X_{k -1}B^aU_j \|_{L^2}^2 \leq C\sum_{i=0}^{k -1} \| X_iU_j \|_{L^2}^2,
\]
from which we infer
\begin{align}
\label{thirdterm}
\begin{split}
- 2^{-2j}\bigl(X_{k -1} B^a U_j , \, X_{k -1} A\partial_x U_j \bigl) \leq C 2^{-2j}\sum_{i=0}^{k -1} \|X_iU_j \|_{L^2}^2 + \frac{1}{32} \|X_{k -1}AU_j \|_{L^2}^2.
\end{split}
\end{align}
Moving on to the fourth term, an integration by parts, along with the Young inequality, yields
\[
2^{-2j}|\bigl(X_{k -1} U_j , \, X_{k -1} A^2 \partial_{xx} U_j \bigl)| \leq \frac{C}{\eps^{q_k }} \|X_{k -1} U_j \|_{L^2}^2+ C\eps^{q_k } \|X_{k -1}A^2U_j \|_{L^2}^2.
\]
Finally, in a similar fashion, we control the last term in the following way
\[
2^{-2j}| \bigl(X_{k -1} U_j , \, X_{k -1} A B^a\partial_x U_j \bigl) | \leq \frac{C}{\eps^{q_k }} \|X_{k -1} U_j \|_{L^2}^2+ C\eps^{q_k } 2^{-2j}\|X_{k -1}AB^a U_j \|_{L^2}^2.
\]
Since we are working in high frequencies, the term $2^{-2j}\|X_{k -1}U_j \|_{L^2}^2$, appearing in the sum of inequality \eqref{thirdterm}, can be absorbed into one of the terms $\|X_{k -1}U_j \|_{L^2}^2$ of the above inequalities. Collecting all the obtained inequalities and multiplying by $\eps^{p_k }$, we have the thesis.
\end{proof}

\subsection*{Case (2): Right} We associate to the node the functional
\[
\Phi^k _{B^a} := 2^{-2j}\eps^{p_k }\bigl(X_{k -1} U_j ,\, X_{k -1} B^a U_j  \bigl)
\]
and the discrepancy
\[
\d_k  = 1.
\]
We have the following lemma.
\begin{lemma}
\label{lemma:right}
    It holds
    \begin{align}
    \label{rightpath1}
    \begin{split}
        \partial_t \Phi^k _{B^a}
        + \frac{\eps^{p_k }}{2}2^{-2j}\|&X_{k -1} B^a U_j \|_{L^2}^2\leq C \eps^{p_k  -q_k }\|X_{k -1} U_j \|_{L^2}^2 + C\eps^{p_k } \sum_{i = 0}^{k -2} \|X_j U_j \|_{L^2}^2 \\ 
        &+\eps^{p_k  + q_k }2^{-2j} \|X_{k -1}B^aA U_j \|_{L^2}^2 + \eps^{p_k  + q_k }2^{-4j}\|X_{k -1}(B^a)^2U_j \|_{L^2}^2.
    \end{split}
    \end{align}   
\end{lemma}
\begin{proof}
    The proof works exactly in the same way as the one of Lemma \ref{lemma:left}. We start with the differential equality
    \begin{align}
    \label{eq:right}
    \begin{split}
        &2^{-2j}\partial_t \bigl(X_{k -1} U_j ,\, X_{k -1} B^a U_j  \bigl) + 2^{-2j} \|X_{k -1} B^a U_j \|_{L^2}^2   \\
        &\quad\quad\quad\quad=-2^{-2j} \bigl(X_{k -1} B^s U_j , \, X_{k -1} B^a U_j \bigl) \\
        &\quad\quad\quad\quad\quad- 2^{-2j}\bigl(X_{k -1} U_j , \, X_{k -1} B^a B^s U_j \bigl) \\
        &\quad\quad\quad\quad\quad- 2^{-2j}\bigl(X_{k -1} A \partial_x U_j , \, X_{k -1} B^a U_j \bigl) \\
        &\quad\quad\quad\quad\quad- 2^{-2j}\bigl(X_{k -1} U_j , \, X_{k -1} B^a A \partial_x U_j \bigl) \\
        &\quad\quad\quad\quad\quad-2^{-2j} \bigl(X_{k -1} U_j , \, X_{k -1} (B^a)^2 U_j \bigl).
    \end{split}
    \end{align}
    The first two terms on the right-hand side of \eqref{eq:right} can be controlled similarly as in the previous lemma. Using $\rank(M_A) = r$ and the Cayley–Hamilton theorem, we arrive at
       \begin{align*}
    \begin{split} 
    2^{-2j}\bigl(X_{k -1} A \partial_x U_j , \, X_{k -1} B^a U_j \bigl)&\leq C 2^{-j} \|X_{k -1}  U_j\|_{L^2}\|X_{k -1} B^a U_j\|_{L^2}\\
    &\leq C 2^{-j}\sum_{i=0}^{k-1}\| X_i  U_j\|_{L^2} \|X_{k -1} B^a U_j\|_{L^2}\\
    &\leq \frac{1}{8} 2^{-2j}\|X_{k -1} B^a U_j\|_{L^2}^2+C \sum_{i=0}^{k-1}\| X_i  U_j\|_{L^2}^2.
        \end{split}
    \end{align*}
    We bound the fourth term in the following way:
    \[
    2^{-2j}| \bigl(X_{k -1} U_j , \, X_{k -1} B^a A \partial_x U_j \bigl)| \leq \frac{C}{\eps^{q_k }}\|X_{k -1}U_j \|_{L^2}^2 + 2^{-2j}\eps^{q_k }\|X_{k -1}B^aAU_j \|_{L^2}^2.
    \]
    Similarly, the last term satisfies
     \[
     2^{-2j}|\bigl(X_{k -1} U_j , \, X_{k -1} (B^a)^2 U_j \bigl)| \leq \frac{C}{\eps^{q_k }}\|X_{k -1}U_j \|_{L^2}^2 + 2^{-4j}\eps^{q_k }\|X_{k -1}(B^a)^2U_j \|_{L^2}^2.
    \qedhere\]
    Substituting the above estimates into \eqref{eq:right} and multiplying this by $\eps^{p_k}$, we prove \eqref{rightpath1}.
\end{proof}

\subsection*{Case (3): Left AND Right}\label{ssmix} This case is the most delicate. A naive approach to recover the norms of both $X_{k -1}AU_j $ and $X_{k -1}B^aU_j $ could be summing up the two equalities \eqref{eq:left} and \eqref{eq:right} obtained in the previous lemmas. However, forgetting the multiplicative factor $2^{-2j}$ for the moment, we have
\begin{align*}
    \begin{split}
        &\partial_t\bigg( \bigl(X_{k -1} U_j ,\, X_{k -1} A\partial_x U_j  + X_{k -1} B^a U_j  \bigl) \bigg) \\
        &\qquad =- \|X_{k -1} B^a U_j \|_{L^2}^2 - \| X_{k -1} A\partial_x U_j \|_{L^2}^2 - 2 \bigl(X_{k -1} B^a U_j , \, X_{k -1} A\partial_x U_j \bigl) + \text{ additional terms. }
    \end{split}
\end{align*}
We have obtained a perfect square. In particular, we are only able to recover the norm of the sum $\|X_{k -1}( A\partial_x + B^a) U_j \|_{L^2}^2$. It is not difficult to see that, using $\Phi_A^k $ and $\Phi_{B^a}^k $, there is no way to recover the sum of the norms. A new functional is needed, with the purpose of canceling out the scalar product appearing on the right-hand side. We will consider two cases, depending on whether the path  extends in the right direction or in the left direction.

\medskip
\noindent
\textbf{Case 1:} In this case, we assume that the tree extends following the $X_{k -1}B^a$ node. More precisely, we make the following assumption.
\begin{assumption}
\label{ass1}
     There exist $\alpha_i, \beta_i$, $i = 0, \ldots, k -1$, such that 
        \[
        X_{k -1}A^2U_j  = \sum_{i=0}^{k -1} \alpha_i X_iU_j  \quad \text{ and } \quad
        X_{k -1}AB^aU_j  = \sum_{i=0}^{k -1} \beta_j X_jU_j .
        \]
\end{assumption}
This guarantees that the node $X_{k -1}A$ is a ``dead end", in that there is no additional dissipation to be gained by following the path further. Actually, we need to ask a little more, namely, that $X_{k -1}A^2$ and $X_{k -1}AB^a$ can be controlled in terms of the nodes \emph{preceding} $X_{k -1}A$.
At this point, we introduce the crucial hypothesis that allows us to work in the mixed case.
\begin{assumption}
\label{ass2}
    There exists $m \in \R$ such that
        \begin{equation}
        \label{cancellation11}
            \bigl(X_{k -1} A \partial_x U_j , \, - X_{k -1} B^a U_j  + m X_{k -1} B^a A^2 U_j \bigl) = 0,
        \end{equation}
\end{assumption}
If \eqref{cancellation11} holds, we can proceed, and we associate to the left node the functional 
\[
\Tilde{\Phi}^k _A := \Phi^k _A + m\eps^{p_k }2^{-2j}\bigl( X_{k -1} A U_j ,\,  X_{k -1} B^a A U_j  \bigl),
\]
and to the right node the functional
\[
\Tilde{\Phi}^k _{B^a} := \Phi^k _{B^a} + m\eps^{p_k }2^{-2j}\bigl( X_{k -1} A U_j ,\,  X_{k -1} B^a A U_j  \bigl).
\]
To both nodes, we associate a discrepancy number 
\[
\d_k  = 1.
\]
We need one final assumption before stating the central result for the mixed case. \begin{assumption}
\label{ass3}
    There exists $\gamma_i$, $i = 0, \ldots, k +1$, such that
        \[
        X_{k -1}B^aAB^aU_j  =
        \gamma_{k +1}X_{k -1}B^aAU_j  + \gamma_k  X_{k -1}B^aU_j  + \sum_{i=0}^{k -1} \gamma_i X_iU_j .
        \]
\end{assumption}
This request is technical. In order to absorb the norms on the right-hand side, we need a control on the norm $\|X_{k -1}B^aAB^a\|_{L^2}^2$ (a node at the level $k +2$) in terms of the norms up to level $k +1$.
Defining
\[
\Phi^k _m := \Tilde{\Phi}^k _A + \Tilde{\Phi}^k _{B^a},
\]
we have the following result.
\begin{lemma}
\label{lemma:magic}
    Let $X_{k -1}$ be a node at the level $k -1$ and let \eqref{eq:rankmix} and $\rank(M_A) > r$  hold. Let also Assumptions \ref{ass1}, \ref{ass2} and \ref{ass3} be satisfied. There exists a suitably large integer $j_0$ such that for all $j\geq j_0$, 
    \begin{align*}
       \partial_t \Phi^k _m &+ \frac{\eps^{p_k }}{4}2^{-2j} \|X_{k -1} B^a U_j \|_{L^2}^2 + \frac{\eps^{p_k }}{4}\| X_{k -1} A U_j \|_{L^2}^2 \\
        &\leq \eps^{p_k  + q_k }2^{-2j}\|X_{k -1}B^aA U_j \|_{L^2}^2 + \eps^{p_k  + q_k }2^{-4j}\|X_{k -1}(B^a)^2U_j \|_{L^2}^2 + C\eps^{p_k -q_k }\sum_{i=0}^{k -1}\|X_i U_j \|_{L^2}^2.
    \end{align*}
\end{lemma}

\begin{proof}[Proof of Lemma \ref{lemma:magic}]
A direct computation yields
     \begin{align}
     \label{eq:fullmix}
    \begin{split}
        2^{-2j}\partial_t \bigl( X_{k  - 1} A U_j ,\,  X_{k  - 1} B^a A U_j  \bigl) =
        &- 2^{-2j}\bigl( X_{k  - 1} A U_j , \,  X_{k  - 1} B^a A^2 \partial_x U_j \bigl) \\
        &- 2^{-2j}\bigl( X_{k  - 1} A B^a U_j , \,  X_{k  - 1} B^a A U_j \bigl) \\
        &- 2^{-2j}\bigl( X_{k  - 1} A B^s U_j , \,  X_{k  - 1} B^a A U_j \bigl) \\
        &- 2^{-2j}\bigl( X_{k  - 1} A U_j , \,  X_{k  - 1} B^a A B^s U_j \bigl) \\  
        &- 2^{-2j}\bigl( X_{k  - 1} A^2 \partial_x U_j , \,  X_{k  - 1} B^a A U_j \bigl) \\
        &- 2^{-2j}\bigl( X_{k  - 1} A U_j , \,  X_{k  - 1} B^a A B^a U_j \bigl).  
    \end{split}
    \end{align}
    In light of Assumption \ref{ass1}, the first term on the right-hand side will cancel out with the scalar products appearing when summing up \eqref{eq:left} with \eqref{eq:right}. The second and fifth terms can be controlled thanks to Assumption \ref{ass2}. Indeed,
    \begin{align*}
    2^{-2j}\bigl( X_{k  - 1} A B^a U_j , \,  X_{k  - 1} B^a A U_j \bigl) &= 2^{-2j}\sum_{j=0}^{k  - 1}  \beta_j\bigl(X_j U_j , X_{k  - 1} B^a A U_j \bigl) \\
    &\leq \frac{C2^{-2j}}{\eps^{q_k }}\sum_{i=0}^{k  - 1} \|X_i U_j \|_{L^2}^2 + \eps^{q_k }2^{-2j}\|X_{k  - 1}B^aAU_j \|_{L^2}^2,
    \end{align*}
    and
    \begin{align*}
    2^{-2j}\bigl( X_{k  - 1} A^2\partial_x U_j , \,  X_{k  - 1} B^a A U_j \bigl) &= 2^{-2j}\sum_{i=0}^{k  - 1}  \alpha_i \bigl(X_i \partial_x U_j , X_{k  - 1} B^a A U_j \bigl) \\
    &\leq \frac{C}{\eps^{q_k }}\sum_{i=0}^{k  - 1} \|X_i U_j \|_{L^2}^2 + \eps^{q_k }2^{-2j}\|X_{k  - 1}B^aAU_j \|_{L^2}^2.
    \end{align*}
    The third and fourth terms can be controlled thanks to the presence of $B^s U_j $:
    \begin{align*}
        -&\bigl( X_{k  - 1} A B^s U_j , \,  X_{k  - 1} B^a A U_j \bigl) -\bigl( X_{k  - 1} A U_j , \,  X_{k  - 1} B^a A B^s U_j \bigl) \\
    &\leq
    \eps^{q_k }\|X_{k  - 1}AU_j \|_{L^2}^2 + 
    \eps^{q_k }\|X_{k  - 1}B^aAU_j \|_{L^2}^2 + \frac{C}{\eps^{q_k }}\|B^sU_j \|_{L^2}^2.
    \end{align*}
    Finally, we use Assumption \ref{ass3} to deal with the last term. We have
    \begin{align*}
        &2^{-2j}\bigl( X_{k  - 1} A U_j , \,  X_{k  - 1} B^a A B^a U_j \bigl) \\
        &= 2^{-2j}\sum_{i=0}^{k  - 1} \gamma_i \bigl(X_{k  - 1}AU_j , X_i U_j \bigl) + \gamma_k 2^{-2j} \bigl(X_{k  - 1}AU_j , X_{k  - 1}B^aU_j \bigl) \\ 
        &\quad+ \gamma_{k +1} 2^{-2j}\bigl(X_{k  - 1}AU_j , X_{k  - 1}B^aA U_j \bigl)  \\
        &\leq \frac{C}{\eps^{q_k }}2^{-2j}\sum_{i=0}^{k  - 1}\|X_iU_j \|_{L^2}^2 + \left(\eps^{q_k }2^{-2j} + C2^{-2j} +\frac{C 2^{-j}}{2}\right) \|X_{k  - 1}AU_j \|_{L^2}^2 \\ &\quad+ \frac{C}{2}2^{-3j}\|X_{k  - 1} B^a U_j \|_{L^2}^2 + \frac{1}{2}\cdot 2^{-3j}\|X_{k  - 1}B^aAU_j \|_{L^2}^2.
    \end{align*}
    The control for the derivatives of the other terms appearing in $\Phi_m^k $ is precisely the same as Lemma \ref{lemma:left} and Lemma \ref{lemma:right}. Collecting the latter inequalities, as well as the inequalities obtained along this proof, and recalling assumptions \ref{ass2} and \ref{ass3}, we arrive at
    \begin{align*}
        2^{-2j}\partial_t \Phi_m^k  &+ \left(\frac{1}{2} - \frac{C}{2} 2^{-j}\right) 2^{-2j}\|X_{k  - 1} B^a U_j \|_{L^2}^2 + \left(\frac{1}{2} - \frac{C}{2}2^{-j}\right) \|X_{k  - 1}AU_j \|_{L^2}^2 \\
        &\leq \eps^{q_k }2^{-2j}\|X_{k  - 1}B^aA U_j \|_{L^2}^2 + \frac{1}{2}\cdot2^{-3j}\|X_{k  - 1}B^aAU_j \|_{L^2}^2 \\
        &+ \eps^{q_k }2^{-4j}\|X_{k  - 1}(B^a)^2U_j \|_{L^2}^2 + C\eps^{p_k -q_k }\|X_{k  - 1}U_j \|_{L^2}^2 + C\sum_{i=0}^{k -2}\|X_i U_j \|_{L^2}^2.
    \end{align*}
    Choosing $j_0$ large enough such that
    \[
    \frac12 - \frac{C}{2}2^{-j_0} \geq \frac{1}{4},
    \]
    we attain the desired inequality.
\end{proof}
\begin{remark}
    Here, we needed the frequency threshold $j_0$ to be large enough in order to close the estimates. Note that $U_j^h=0$ for $j\leq -1$. To deal with $0\leq j\leq j_0-1$, which is composed of a finite number of frequencies, one can use the computations from Section \ref{sec:3} (for instance $\eqref{LLY}$) and deduce exponential decay as $2^j$ is bounded from below and above in this frequency region.
\end{remark}

\medskip
\noindent
\textbf{Case 2:} The second case is similar to the first one. Now it is $X_{k  - 1}B^a$ that is the end node. Hence, we have the following set of assumptions.
\begin{assumption}
    \label{ass4}
     There exist $\alpha_i, \beta_i$, $i = 0, \ldots, k -1$, such that 
        \[
        X_{k  - 1}(B^a)^2U_j  = \sum_{i=0}^{k  - 1} \alpha_i X_iU_j  \quad \text{ and } \quad
        X_{k  - 1}B^aAU_j  = \sum_{i=0}^{k  - 1} \beta_i X_iU_j .
        \]
\end{assumption}
\begin{assumption}
    \label{ass5}
    There exists $m \in \R$ such that
        \begin{equation}
        \label{cancellation12}
            \bigl(m X_{k  - 1} AB^aA \partial_x U_j  -  X_{k  - 1} A \partial_x U_j , \, X_{k  - 1} B^a U_j  \bigl) = 0.
        \end{equation}
\end{assumption}

\begin{assumption}
     \label{ass6} 
     There exist $\gamma_i$, $i = 0, \ldots, k $, such that
        \[
        X_{k  - 1}A(B^a)^2U_j  = \gamma_k  X_{k  - 1}AU_j  + \sum_{i=0}^{k  - 1} \gamma_i X_iU_j .
        \]        
\end{assumption}
Note that, with respect to the previous case, we ask that $\gamma_{k +1} = 0$. 
We then associate to the left node the functional 
\[
\Tilde{\Phi}^k _A := \Phi^k _A + m\eps^{p_k }2^{-2j}\bigl( X_{k  - 1} AB^a U_j ,\,  X_{k  - 1} B^a U_j  \bigl),
\]
and to the right node the functional
\[
\Tilde{\Phi}^k _{B^a} := \Phi^k _{B^a} + m\eps^{p_k }2^{-2j}\bigl( X_{k  - 1} AB^a U_j ,\,  X_{k  - 1} B^a U_j  \bigl).
\]
In this case, we associate to both nodes the discrepancy
\[
\d_k  = 0.
\]
Hence, the final functional reads 
\[
\Phi^k _m := \Tilde{\Phi}^k _A + \Tilde{\Phi}^k _{B^a}.
\]
We are now in a position to state the lemma.
\begin{lemma}
\label{lemma:magic2}
    Let $X_{k  - 1}$ be a node at the level $k -1$ and let \eqref{eq:rankmix} and $\rank(M_A) > r$ hold. Let also Assumptions \ref{ass4}, \ref{ass5} and \ref{ass6} be satisfied. 
    Then,
    \begin{align*}
       \partial_t \Phi^k _m &+ \frac{\eps^{p_k }}{4}2^{-2j} \|X_{k  - 1} B^a U_j \|_{L^2}^2 + \frac{\eps^{p_k }}{4}\| X_{k  - 1} A U_j \|_{L^2}^2 \\
        &\leq \eps^{p_k  + q_k }\|X_{k  - 1}A^2 U_j \|_{L^2}^2 + \eps^{p_k  + q_k }2^{-2j}\|X_{k  - 1}AB^aU_j \|_{L^2}^2 + C\eps^{p_k -q_k }\sum_{i=0}^{k  - 1}\|X_i U_j \|_{L^2}^2.
    \end{align*}
\end{lemma}
\begin{proof}[Proof of Lemma \ref{lemma:magic2}]
The counterpart of Equality \eqref{eq:fullmix} reads
     \begin{align}
    \begin{split}
        2^{-2j}\partial_t \bigl( X_{k-1} A^2 U_j ,\,  X_{k-1} B^a U_j  \bigl) =
        &- 2^{-2j}\bigl( X_{k-1} AB^a U_j , \,  X_{k-1} B^a A \partial_x U_j \bigl) \\
        &- 2^{-2j}\bigl( X_{k-1} A (B^a)^2 U_j , \,  X_{k-1} B^a U_j \bigl) \\
        &- 2^{-2j}\bigl( X_{k-1} AB^a B^s U_j , \,  X_{k-1} B^a U_j \bigl) \\
        &- 2^{-2j}\bigl( X_{k-1} AB^a U_j , \,  X_{k-1} B^a B^s U_j \bigl) \\  
        &- 2^{-2j}\bigl( X_{k-1} AB^aA \partial_x U_j , \,  X_{k-1} B^a U_j \bigl) \\
        &- 2^{-2j}\bigl( X_{k-1} AB^a U_j , \,  X_{k-1} (B^a)^2 U_j \bigl).  
    \end{split}
    \end{align}
In this case, by Assumption \ref{ass4}, the fifth term cancels out. The rest of the proof is the same as that of Lemma \ref{lemma:magic}, the only modification being how we treat the second term. Indeed, now we have
\begin{equation*}
\begin{aligned}
    - &2^{-2j}\bigl( X_{k-1} A (B^a)^2 U_j , \,  X_{k -1} B^a U_j \bigl) \\ 
    &= -2^{-2j}\gamma_k \bigl(X_{k  - 1}AU_j, X_{k -1}B^aU_j\bigl) +2^{-2j} \sum_{i=0}^{k  - 1}\gamma_i\bigl(X_iU_j , X_{k -1}B^aU_j\bigl) \\
    &\leq 2^{-j}C\|X_{k -1}AU_j\|_{L^2}^2 + 2^{-3j}C \|X_{k -1}B^aU_j\|_{L^2}^2 + C \sum_{i=0}^{k -1}\|X_i U_j\|_{L^2}^2.
\end{aligned}
\end{equation*}
The proof is concluded in the same exact way as in the previous case.
\end{proof}

\subsection*{Case (4): Left OR Right} As described earlier, we simply choose $X_{k -1}A$. Accordingly, the functional is $\Phi_A^k $ and the discrepancy is $\d_k  = 0$.

\subsection*{Case (5): End of the path} If the rank of the matrices $M_A$ and $M_{B^a}$ is equal to $r$, it indicates that no additional dissipation can be achieved by progressing further down the tree. At this point, we check the value of $r$. If $r = n$, we have successfully recovered the dissipation for all the variables appearing in the system. Otherwise, we move to the next node (i.e., the node on the right) at the same level $k -1$. Moreover, if $r<n$, there always exists a node such that one of the previous three cases applies. Indeed, if this were not the case, then every term of $B^s(A\partial_x + B^a)^k  U_j$ would be a combination of previous terms. Consequently, it is immediate to check that the inhomogeneous Kalman condition \eqref{condition:span} would fail to hold, contradicting the assumption at the beginning of this section.

\subsection{Cancellations} The main advantage of the algorithm we have described is that it allows us to automatically discover conditions under which cancellations occur. In Lemma \ref{lemma:magic}, we have already seen how canceling a term in a differential equality can help to recover additional dissipation. Now we focus on two additional situations in which we can improve our decay estimates. We will work within the hypothesis
\begin{equation}
    \label{cond}
    X_{k -1}(B^a)^2U_j = \sum_{i=0}^{k } \delta_i X_i U_j
\end{equation}
where the $\delta_i$ are constants.
We remark that this assumption is not necessary to improve the decay rate. Although it would be possible to remove it, it greatly simplifies the analysis and the exposition.

\subsubsection{Cancellation in the right path} When we take the right path moving from one level of the tree to the next one, we have the differential equality \eqref{eq:right}.
The terms containing the derivative on the right-hand side require us to weight the estimate by $2^{-2j}$, effectively losing a derivative, as we can only recover the norm $2^{-2j}\|X_{k -1}B^a U_j \|_{L^2}^2$. However, there are situations in which these terms simply cancel out. If this is the case, \eqref{eq:right} simplifies into
\begin{align*}
        \partial_t \bigl(X_{k -1} U_j ,\, X_{k -1} B^a U_j  \bigl) + \|X_{k -1} B^a U_j \|_{L^2}^2 =  
        &- \bigl(X_{k -1} B^s U_j , \, X_{k -1} B^a U_j \bigl) \\
        &- \bigl(X_{k -1} U_j , \, X_{k -1} B^a B^s U_j \bigl) \\
        &- \bigl(X_{k -1} U_j , \, X_{k -1} (B^a)^2 U_j \bigl).
\end{align*}
At that point, by hypothesis \eqref{cond}, we have the following
\begin{align*}
        \partial_t \bigl(X_{k -1} U_j ,\, X_{k -1} B^a U_j  \bigl) + \frac{3}{4}\|X_{k -1} B^a U_j \|_{L^2}^2 \leq  
        C\|B^s U_j \|_{L^2}^2 + \frac{C}{\eps^{p_k }}\sum_{i=0}^{k -1}\|X_iU_j \|_{L^2}^2.
\end{align*}
In this specific case, we associate a discrepancy $\d_k  = 0$. 



\subsubsection{Cancellation in the mixed path} Another cancellation effect may occur in the mixed path. This is by far the most interesting situation, as it appears in several physical models that exhibit different behaviors depending on the parameters of the system. A classical example is the Timoshenko system, which is known to have better decay properties when the so-called \emph{equal wave speed} condition holds. As we will see, the latter appears naturally as a cancellation in the mixed path functional. 

Assume we are in Case 1 and consider the functional 
\[
\Phi = \bigl(X_{k -1} U_j ,\, X_{k -1} B^a U_j  \bigl) + m\bigl( X_{k -1} A U_j ,\,  X_{k -1} B^a A U_j  \bigl).
\] 
If Assumption \ref{ass2} holds, taking the time derivative, we obtain (see equation \eqref{eq:fullmix})
\begin{align*}
\partial_t \Phi + \|X_{k -1} B^a U_j \|_{L^2}^2 =  
&- \bigl(X_{k -1} U_j  - m X_{k -1}A^2 U_j , \, X_{k -1} B^a A \partial_x U_j \bigl) \\
&+ \text{ lower order terms }.
\end{align*}
Note that $m$ here is fixed by Assumption \ref{ass2}. If the additional cancellation condition
\begin{equation}
\label{cancellation21}
\bigl(X_{k -1} U_j  - m X_{k -1}A^2 U_j , \, X_{k -1} B^a A \partial_x U_j \bigl) = 0
\end{equation}
holds, then we can redefine $\Tilde{\Phi}^k _{B^a}$ in Lemma \ref{lemma:magic} to be simply
\[
\Tilde{\Phi}^k _{B^a} = \eps^{p_{k +1}}\Phi.
\]
In this case, the discrepancy associated with both nodes is \[\d_k  = 0.\]
By the same token, in Case 2 we have the similar condition
\begin{equation}
\label{cancellation22}
\bigl(X_{k -1} U_j  - m X_{k -1}AB^a U_j , \, X_{k -1} B^a A \partial_x U_j \bigl) = 0.
\end{equation}
In general, since $m$ is fixed by assumptions \ref{ass2} or \ref{ass5}, Equations \eqref{cancellation21}-\eqref{cancellation22} can be read as restrictions on the coefficients of the system which, if satisfied, provide better decay rates. However, note that both conditions are also satisfied when $X_{k -1}B^aAU_j  = 0$. This is very relevant in concrete physical settings when $A$ and $B^a$ are often sparse. 
We have the following result.
\begin{lemma}
    Let the assumptions of Lemma \ref{lemma:magic} hold. Furthermore, assume that \eqref{cond} and \eqref{cancellation21} hold. Then,
    \begin{align*}
        \partial_t \Phi^{k }_m &+ \frac{\eps^{p_k }}{4}\|X_{k -1}AU_j\|_{L^2}^2 + \frac{\eps^{p_{k +1}}}{2}\|X_{k -1}B^aU_j\|_{L^2}^2 \\
        &\qquad\leq C\eps^{p_k  - q_k }\sum_{j=0}^{k -1}\|X_jU_j\|_{L^2}^2 + \eps^{p_{k +1}+q_{k +1}}\|X_{k -1}B^aAU_j\|_{L^2}^2.
    \end{align*}
\end{lemma}

\begin{proof}
    Keeping in mind our assumptions, a simple application of the Young inequality to every term of \eqref{eq:fullmix}, in conjunction with \eqref{leftpath}, yields
    \begin{align}
    \label{phiA}
    \begin{split}
            \partial_t \Tilde{\Phi}^k_A + \frac{\eps^{p_k }}{2}\|X_{k -1}AU_j\|_{L^2}^2 &\leq \eps^{p_k }2^{-2j}\|X_{k -1}B^aU_j\|_{L^2}^2 \\
            &\quad+ C\eps^{p_k  - q_k }\sum_{j=0}^{k -1}\|X_jU_j\|_{L^2}^2 + \eps^{p_k }2^{-2j}\|X_{k -1}B^aAU_j\|_{L^2}^2.
    \end{split}
    \end{align}
    Besides, in light of Assumption \ref{ass2} and \eqref{cancellation21}, we have
    \begin{align*}
        \partial_t \Phi + \|X_{k -1}B^aU_j\|_{L^2}^2 &= - \bigl(X_{k -1} B^s U_j, \, X_{k -1} B^a U_j\bigl) - \bigl(X_{k -1} U_j, \, X_{k -1} B^a B^s U_j\bigl) \\
        &- m\bigl( X_{k -1} A B^s U_j, \,  X_{k -1} B^a A U_j\bigl) - m\bigl( X_{k -1} A U_j, \,  X_{k -1} B^a A B^s U_j\bigl) \\ 
        &- \bigl(X_{k -1} U_j, \, X_{k -1} (B^a)^2 U_j\bigl) - m\bigl( X_{k -1} A B^a U_j, \,  X_{k -1} B^a A U_j\bigl) \\
        &-m \bigl( X_{k -1} A U_j, \,  X_{k -1} B^a A B^a U_j\bigl).
    \end{align*}
    The first two terms on the right-hand side contain $B^sU$, so they can be easily bounded as in Lemma \ref{lemma:right}, that is
    \begin{align*}
     &- \bigl(X_{k -1} B^s U_j, \, X_{k -1} B^a U_j\bigl) - \bigl(X_{k -1} U_j, \, X_{k -1} B^a B^s U_j\bigl) \\ &\qquad\leq \frac{C}{\eps^{q_{k +1}}}\|B^sU_j\|_{L^2}^2 + \eps^{q_{k +1}}\|X_{k -1}B^aU_j\|_{L^2}^2 + C\|X_{k -1}U_j\|_{L^2}^2.
    \end{align*}
    Similarly, the second line is controlled as
    \begin{align*}
     &- m\bigl( X_{k -1} A B^s U_j, \,  X_{k -1} B^a A U_j\bigl) - m\bigl( X_{k -1} A U_j, \,  X_{k -1} B^a A B^s U_j\bigl) \\ &\qquad\leq \frac{C}{\eps^{q_{k +1}}}\|B^sU_j\|_{L^2}^2 + \eps^{q_{k +1}}\|X_{k -1}B^aAU_j\|_{L^2}^2 + \eps^{q_{k +1}}\|X_{k -1}AU_j\|_{L^2}^2.
    \end{align*}
    Next, exploiting \eqref{cond}, we have
    \begin{align*}
     \bigl(X_{k -1} U_j, \, X_{k -1} (B^a)^2 U_j\bigl) &= \delta_k  \bigl(X_{k -1}U_j, X_{k -1}B^aU_j\bigl) + \sum_{i=0}^{k -1} \delta_i \bigl(X_{k -1}U_j, X_iU_j\bigl) \\
     &\leq \eps^{q_{k +1}}\|X_{k -1}B^aU_j\|_{L^2}^2 + \frac{C}{\eps^{q_{k +1}}}\sum_{i=0}^{k -1}\|X_iU_j\|_{L^2}^2.
    \end{align*}
    The final two terms are treated as in the proof of Lemma \ref{lemma:magic}, except this time we do not have the weight $2^{-2j}$. We have
    \[
   m \bigl( X_{k -1} A B^a U_j, \,  X_{k -1} B^a A U_j\bigl) \leq \frac{C}{\eps^{q_{k +1}}}\sum_{i=0}^{k -1} \|X_i U_j\|_{L^2}^2 + \eps^{q_{k +1}}\|X_{k -1}B^aAU_j \|_{L^2}^2
    \]
    and 
    \begin{align*}
   &m \bigl( X_{k -1} A U_j, \,  X_{k -1} B^a A B^a U_j\bigl) \\ &\quad\leq \frac{C}{\eps^{q_{k +1}}}\sum_{i=0}^{k -1}\|X_jU_j\|_{L^2}^2 + \|X_{k -1}AU_j\|_{L^2}^2 + \frac14\|X_{k -1}B^aU_j\|_{L^2}^2 +  \eps^{q_{k +1}}\|X_{k -1}B^aAU_j\|_{L^2}^2.
    \end{align*}
    Collecting together the obtained inequalities we find
    \begin{align*}
        \partial_t \Tilde{\Phi}^k_{B^a} + \frac{\eps^{p_{k +1}}}{2}\|X_{k -1}B^aU_j\|_{L^2}^2 &\leq C\eps^{p_{k +1} - q_{k +1}}\sum_{i=0}^{k -1}\|X_i U_j\|_{L^2}^2 \\
        &+ \eps^{p_{k +1}}\|X_{k +1}AU_j\|_{L^2}^2
        + \eps^{p_{k +1} + q_{k +1}}\|X_{k -1}B^a A U_j\|_{L^2}^2.
    \end{align*}
    Summing up to \eqref{phiA} we have the thesis.
\end{proof}
We limit ourselves to stating the version of the above lemma for the second set of assumptions, as the proof works the same way. The definition of $\Phi^{k }_m$ does not change.
\begin{lemma}
    Let the assumptions of Lemma \ref{lemma:magic2} hold. Furthermore, assume that \eqref{cancellation22} holds. Note that \eqref{cond} holds automatically thanks to Assumption \ref{ass4}. Then,
    \begin{align*}
        \partial_t \Phi^{k }_m &+ \frac{\eps^{p_k }}{4}\|X_{k -1}AU_j\|_{L^2}^2 + \frac{\eps^{p_{k +1}}}{2}\|X_{k -1}B^aU_j\|_{L^2}^2 \\
        &\qquad\leq C\eps^{p_k  - q_k }\sum_{j=0}^{k -1}\|X_jU_j\|_{L^2}^2 + \eps^{p_{k +1}+q_{k +1}}\|X_{k -1}A^2U_j\|_{L^2}^2 + \eps^{p_k +q_k }\|X_{k -1}AB^aU_j\|_{L^2}^2.
    \end{align*}
\end{lemma}
In this case, we do not have an improvement of the discrepancy, but we still obtain better regularity since we now recover the control of the norm $\|X_{k -1}B^aU_j\|_{L^2}^2$.
\begin{remark}
    Note that the coefficients are now unbalanced, being $\eps^{p_k }$ and $\eps^{p_{k +1}}$. This was necessary in the proof to be able to absorb the term $\eps^{p_{k +1}}\|X_{k -1}AU_j\|_{L^2}^2$ of the last inequality with the norm recovered at \eqref{phiA}. Clearly, if there is a cancellation and $X_{k -1}B^a$ has a child node, we will now need to multiply the associated functional by $\eps^{p_{k +2}}$ instead of $\eps^{p_{k +1}}$. This, however, has no effect on the previous results and discussion.
\end{remark}

\section{Low-frequency Analysis}\label{sec:ImproveLF}

\noindent
We now investigate the behavior of the functionals introduced in the previous section in the low-frequency regime. In this regime, the roles of $A$ and $B^a$ are somewhat inverted. In particular, when we add a left node to the tree, we generally lose regularity. However, there is no true symmetry with the high-frequency setting: while in high frequencies the "good" functional $\Phi^{k }_A$ needs to be multiplied by $2^{-2j}$, in low frequencies the good functional $\Phi^{k }_{B^a}$ need not be. This creates an imbalance in the differential inequalities, particularly when dealing with the mixed term, leading to technical barriers that we were unable to overcome (at least, not in a meaningful way, see Remark \ref{rem:leftright}). Accordingly, we will work under the following assumption.
\begin{assumption}
\label{assump:low}
    The "Left AND Right" case never happens at low frequencies.
\end{assumption}

In this section, we always use the low-frequency variable $U^{\ell}=\mathcal{F}^{-1}(\mathbf{1}_{|\xi|\leq 1}\widehat{U})$ and simplify the notation by omitting the superscript $^\ell$. It is important to note that we are still able to treat the case in which mixed terms appear using the functional derived from the Kalman condition in Section \ref{sec:2} (see, e.g., Example \ref{ex:3by3}).
We now proceed with the analysis of the four remaining cases.

\subsection*{Case (1): Left} We will associate the term
\[
\Phi^{k }_A := \eps^{p_k }\bigl(X_{k -1} U_j,\, X_{k -1} A\partial_x U_j \bigl),
\]
and add to the node a discrepancy \[\d_k  = 1.\]
The following lemma holds.
\begin{lemma}
\label{lemma:left_low}
    It holds
    \begin{align}
    \begin{split}
        \partial_t \Phi^{k }_A& + \frac{\eps^{p_k }}{2}2^{2j}\| X_{k -1} A U_j\|_{L^2}^2 \\
        &\leq C \eps^{p_k  - q_k } \|X_{k -1} U_j\|_{L^2}^2 + C\eps^{p_k }\sum_{i = 0}^{k -2} \|X_i U_j\|_{L^2}^2 \\ 
        &\quad+ \eps^{p_k  + q_k } 2^{4j} \|X_{k -1}A^2 U_j\|_{L^2}^2 + \eps^{p_k  + q_k } 2^{2j} \|X_{k -1}AB^aU_j\|_{L^2}^2.
    \end{split}
    \end{align}
\end{lemma}

\begin{proof}
The counterpart to \eqref{eq:left} reads
    \begin{align}
    \label{eq:left_low}
    \begin{split}
        \partial_t \bigl(X_{k -1} U_j,\, X_{k -1} A\partial_x U_j \bigl) + \| X_{k -1} A \partial_x U_j\|_{L^2}^2 =  
        &- \bigl(X_{k -1} B^s U_j, \, X_{k -1} A \partial_x U_j\bigl) \\
        &-\bigl(X_{k -1} U_j, \, X_{k -1} A B^s\partial_x U_j\bigl) \\
        &-\bigl(X_{k -1} B^a U_j, \, X_{k -1} A\partial_x U_j\bigl) \\
        &-\bigl(X_{k -1} U_j, \, X_{k -1} A^2 \partial_{xx} U_j\bigl) \\
        &-\bigl(X_{k -1} U_j, \, X_{k -1} A B^a\partial_x U_j\bigl).
    \end{split}
\end{align}
Here, the support of $\widehat{U}_j$ leads to 
$$
\| X_{k -1} A \partial_x U_j\|_{L^2}^2\geq \frac{9}{16}2^{2j}\|X_{k -1} AU_j\|_{L^2}^2.
$$
The first three terms are controlled exactly as in Lemma \ref{lemma:left}. 
Moving on to the fourth term, we have
\[
|\bigl(X_{k -1} U_j, \, X_{k -1} A^2 \partial_{xx} U_j\bigl)| \leq \frac{C}{\eps^{q_k }} \|X_{k -1} U_j\|_{L^2}^2+ \eps^{q_k } 2^{4j} \|X_{k -1}A^2U_j\|_{L^2}^2.
\]
In a similar fashion, we control the last term in the following way
\[
| \bigl(X_{k -1} U_j, \, X_{k -1} A B^a\partial_x U_j\bigl) | \leq \frac{C}{\eps^{q_k }} \|X_{k -1} U_j\|_{L^2}^2+ \eps^{q_k } 2^{2j}\|X_{k -1}AB^a U_j\|_{L^2}^2.
\]
The conclusion follows.
\end{proof}

\subsection*{Case (2): Right}
We consider the functional
\[
\Phi^{k }_{B^a} := \eps^{p_k }\bigl(X_{k -1} U_j,\, X_{k -1} B^a U_j \bigl).
\]
Moreover, we add a discrepancy $\d_k  = 0$ to the node.
\begin{lemma}
\label{lemma:right_low}
    It holds
    \begin{align}
    \label{rightpath}
    \begin{split}
        \partial_t \Phi^{k }_{B^a}
        +&\frac{3\eps^{p_k }}{4}\|X_{k -1} B^a U_j\|_{L^2}^2\\
        &\leq C \eps^{p_k  -q_k }\|X_{k -1} U_j\|_{L^2}^2 + C\eps^{p_k } \sum_{j = 0}^{k -2} \|X_j U_j\|_{L^2}^2 \\ 
        &\quad+ \eps^{p_k  + q_k } 2^{2j}\|X_{k -1}B^aA U_j\|_{L^2}^2 + \eps^{p_k  + q_k }\|X_{k -1}(B^a)^2U_j\|_{L^2}^2.
    \end{split}
    \end{align}    
\end{lemma}

\begin{proof}
    The counterpart of \eqref{eq:right} reads
    \begin{align}
    \label{eq:right_low}
    \begin{split}
        \partial_t \bigl(X_{k -1} U_j,\, X_{k -1} B^a U_j \bigl) + \|X_{k -1} B^a U_j\|_{L^2}^2 =  
        &-\bigl(X_{k -1} B^s U_j, \, X_{k -1} B^a U_j\bigl) \\
        &-\bigl(X_{k -1} U_j, \, X_{k -1} B^a B^s U_j\bigl) \\
        &-\bigl(X_{k -1} A \partial_x U_j, \, X_{k -1} B^a U_j\bigl) \\
        &-\bigl(X_{k -1} U_j, \, X_{k -1} B^a A \partial_x U_j\bigl) \\
        &-\bigl(X_{k -1} U_j, \, X_{k -1} (B^a)^2 U_j\bigl).
    \end{split}
    \end{align}
    The proof is a direct application of Young's inequality in the same vein as Lemma \ref{lemma:right}.
\end{proof}

\subsection*{Case (4): Left OR Right} In this case, we simply do the opposite as in the high-frequency regime. Specifically, if both $A$ and $B^a$ recover the same dissipation, we add $X_{k -1}B^a$ to the path.

\subsection*{Case (5): End of the Path} The procedure is the same as described in Section \ref{sec:ImproveHF}.

\begin{remark}[Mixed case and low frequencies]
\label{rem:leftright}
    For the sake of completeness, we mention that in the mixed case "Left AND Right" it is still possible to prove a result similar to Lemma \ref{lemma:magic}. Specifically, if Assumption \ref{ass1} holds, and the following two assumptions
    \begin{enumerate}
        \myitem{(2')} \label{ass2prime} There exists $\alpha_i$, $i = 0, \ldots, k -1$, such that 
        \[
        X_{k -1}A^2U_j = \sum_{i=0}^{k -1} \alpha_i X_iU_j
        \]
        and
        \[
        X_{k -1}AB^a = 0.
        \]
        \myitem{(3')} \label{ass3prime} Similarly, we assume
        \[
        X_{k -1}B^aAB^a = 0.
        \]        
    \end{enumerate}
    hold in place of \ref{ass2}-\ref{ass3}, then one can define the functional 
    \[
\Phi^{k }_m := \Tilde{\Phi}^{k }_{A} + \Tilde{\Phi}^{k }_{B^a},
\]
where
    \[
\Tilde{\Phi}^{k }_A := 2^{2j}\Phi^{k }_A + m2^{2j}\bigl( X_{k -1} A U,\,  X_{k -1} B^a A U \bigl)
\]
and
\[
\Tilde{\Phi}^{k }_{B^a} := 2^{2j}\Phi^{k }_{B^a} + m2^{2j}\bigl( X_{k -1} A U,\,  X_{k -1} B^a A U \bigl).
\]
Then, it is possible to show that $\Phi^k _m$ satisfies
\begin{align*}
\partial_t \Phi_k ^m &+\frac{\eps^{p_k }}{2}2^{4j}\|X_{k -1}AU\|_{L^2}^2 + \frac{3\eps^{p_k }}{4}2^{2j}\|X_{k -1}B^aU\|_{L^2}^2 \\ &\leq C\eps^{p_k  - q_k }\|B^sU\|_{L^2}^2 + C \eps^{p_k  - q_k }2^{2j}\sum_{i=0}^{k -1}\|X_i U\|_{L^2}^2+ \eps^{p_k +q_k }2^{4j}\|X_{k -1}B^aAU\|_{L^2}^2.
\end{align*}
This inequality, however, matches poorly with those obtained in Lemma \ref{lemma:left_low} and \ref{lemma:right_low}, due to the weights $2^{4j}$ and $2^{2j}$ in front of the recovered norms. Hence, while for some ad hoc systems we might still be able to cover the mixed case, we have preferred to work within Assumption \ref{assump:low} to make the exposition clearer. 
\end{remark}

\section{The Improved Functional} 
\label{sec:7}

\noindent
We can now collect all the information obtained in the previous two sections to define the Lyapunov functional. Let 
\[
\mathsf{N} = \big\{ X_{k }^i, \, k  = 1, \ldots, K, \, i = 1, \ldots, M_k \big\}
\] 
be the set of nodes chosen with the algorithm. 

\subsection{High frequencies}
With a slight abuse of notation, let us indicate by $\d_{k -1}^i$ the discrepancy of the father node of $X_k ^i$. We define recursively the numbers
\[
\alpha^i_k  = 
\begin{cases}
    \alpha_0 = 0, \\
    \alpha^i_k  = \alpha^i_{k -1} &\text{if } \d^i_{k -1} = 0, \\
    \alpha^i_k  = \alpha^i_{k -1} + 1 &\text{if } \d^i_{k -1} = 1.
\end{cases}
\]
Conventionally, we set $\d_0 = 0$.
At this point, for every fixed $j \geq 0$,
to every node $X_k ^i \in \mathsf{N}$ we associate the functional 
\[
\Psi_k ^i =
2^{-2\alpha_k ^ij}\times
\begin{dcases}
  \Phi^k _A &\text{if left path at } k -1,\\
  \Phi^k _{B^a} &\text{if right path at } k -1, \\
  \Tilde{\Phi}^k _A  &\text{if mixed path at } k -1 \text{ and } X^i_k  = X_{k -1}A, \\
  \Tilde{\Phi}^k _{B^a}  &\text{if mixed path at } k -1 \text{ and } X^i_k  = X_{k -1}B^a.
\end{dcases}
\]
Then, the final Lyapunov functional reads
\[
\L^h_j = \frac12 \|U^h_j\|_{L^2}^2 + \sum_{k =1}^K \sum_{i=1}^{M_k } \Psi_k ^i.
\]
To state our main result, we need one last definition. Let 
\[
(I, L) = \underset{i, k }{\mathrm{argmax}}\, \alpha^i_k .
\]
The variables associated with these indexes are the ones for which we have the largest regularity loss. Accordingly, these nodes dictate the final decay rate.
We define
\[
\widetilde{\alpha}:=
\begin{dcases}
    \alpha^I_L &\text{if we recover } \|X_L^IU^h_j\|_{L^2}^2, \\
    \alpha^I_L + 1 &\text{if we recover } 2^{-2j}\|X^I_LU^h_j\|_{L^2}^2.
\end{dcases}
\]
We then have the following result.
\begin{theorem}[High-frequency decay]
\label{th:improvedh} 
For any $\gamma>0$, if $\widetilde{\alpha}\geq1$, then
\begin{equation}
\label{eq:decayhfimp}
    \|U^h\|_{L^2} \leq C(1+t)^{-\gamma}\|U_0^h\|_{H^{\gamma\widetilde{\alpha}}}.
\end{equation}
If $\widetilde{\alpha}=0$, the high-frequency part of the solution decays exponentially.
\end{theorem}
\begin{proof}
    The proof simply consists
    in applying Lemmas \ref{lemma:left}, \ref{lemma:right} or \ref{lemma:magic} depending on the path defined by the Lyapunov functional. Recall that the integer $j_0$ is given by Lemma \ref{lemma:magic}. Due to the inhomogeneous Kalman rank condition, we have the inequality \eqref{LLY}, which leads to the exponential decay of $\|U_j^h\|_{L^2}^2\lesssim e^{-ct}\|U_{0,j}^h\|_{L^2}^2$ for all $j\leq j_0-1$.  Collecting all the inequalities in Lemmas \ref{lemma:left}, \ref{lemma:right} or \ref{lemma:magic}, and using that $(p_k,q_k)$ is an admissible sequence, for $j\geq j_0$, we arrive at
    \[
    \frac{d}{dt}\L^h_j + \eps^{p_K}\sum_{k =1}^K \sum_{i=1}^{M_k } 2^{-2\gamma_k ^ij} \| X_k ^i U_j^h\|_{L^2}^2 \leq 0,
    \]
    where $\gamma_k ^i$ depends on the choice of $\Psi_k ^i$ and encodes the regularity of the norm recovered by adding $X_k ^i$ to the path. In particular, 
    \[
    \begin{cases}
    \gamma^k _i = \alpha_k ^i &\text{if we recover } \|X_{k }^iU^h_j\|_{L^2}^2, \\
    \gamma^k _i = \alpha_k ^i + 1 &\text{if we recover } 2^{-2j}\|X_{k }^iU^h_j\|_{L^2}^2.
    \end{cases}
    \]
    By the definition of $\widetilde{\alpha}$, we have
    \[
    \frac{d}{dt}\L^h_j + c 2^{-2\widetilde{\alpha} j}\eps^{p_K}\sum_{k =1}^K \sum_{i=1}^{M_k } \| X_k ^i U^h_j\|_{L^2}^2 \leq 0,
    \]
    and \eqref{eq:decayhfimp} follows by the same arguments as Proposition \ref{prophf}.
\end{proof}

\begin{remark}[The role of cancellations]
    In Section \ref{sec:ImproveHF},we have seen how, under certain conditions on the coefficients, cancellations may provide a better decay rate for the solution. This is reflected, in Theorem \ref{th:improvedh}, by the fact that a cancellation on node $X^i_k $ decreases the value of $\gamma^i_k $. Consequently, the final value of $\widetilde{\alpha}$ with cancellations is less than or equal to the value without cancellations. 
\end{remark}

\subsection{Low frequencies} The construction is the same. We define recursively the numbers
\[
\beta^i_k  = 
\begin{cases}
    \beta_0 = 0, \\
    \beta^i_k  = \beta^i_{k -1} &\text{if } \d^i_{k -1} = 0, \\
    \beta^i_k  = \beta^i_{k -1} + 1 &\text{if } \d^i_{k -1} = 1,
\end{cases}
\]
with $\d_0 = 0$ as before.
At this point, for every fixed $j \geq 0$,
to every node $X_k ^i \in \mathsf{N}$ we associate the functional 
\[
\Psi_k ^i =
2^{2\beta_k ^ij}\cdot
\begin{dcases}
  \Phi^k _A &\text{if left path at } k -1,\\
  \Phi^k _{B^a} &\text{if right path at } k -1 .
\end{dcases}
\]
Then, the final Lyapunov functional in low frequencies is
\[
\L^\ell_j = \frac12 \|U^\ell_j\|_{L^2}^2 + \sum_{k =1}^K \sum_{i=1}^{M_k } \Psi_k ^i.
\] 
For 
\[
(I, L) = \underset{i, k }{\mathrm{argmax}}\, \beta^i_k ,
\]
we define, as before,
\[
\widetilde{\beta}:=
\begin{dcases}
    \beta^I_L &\text{if we recover } \|X_L^IU^\ell_j\|_{L^2}^2, \\
    \beta^I_L + 1 &\text{if we recover } 2^{2j}\|X^I_LU^\ell_j\|_{L^2}^2.
\end{dcases}
\]
\begin{theorem}[Low-frequency decay]
\label{th:improvedl}  
 For any $\gamma>0$, if $\widetilde{\beta}\geq1$
\begin{equation}
\label{eq:decayhfimp2}
    \|U^\ell\|_{L^2} \leq C(1+t)^{-\frac{1}{4\widetilde{\beta}}}\|U_0^\ell\|_{L^2}.
\end{equation}
If $\widetilde{\beta}=0$, the low-frequency part of the solution decays exponentially.
\end{theorem}
The proof is the same as in high frequencies, and is therefore omitted.


\section{Cancellations and Rank-One Matrices}
\label{sec:8}

\noindent
In the previous sections, we have witnessed the important role played by cancellations in the high-frequency decay rate of the solution. Specifically, the two sets of conditions we are interested in are \eqref{cancellation11}, \eqref{cancellation12}
and 
\eqref{cancellation21}, \eqref{cancellation22}.
The former two allow us to recover the sum between $A$-norm and $B^a$-norm in the mixed case, and the latter allows us to obtain better regularity from the functional $\Phi^k_m$. It is then of paramount importance to find conditions on the matrices $A, B^a$ and $B^s$ so that these cancellations appear. As a case study, in this section, we consider the case in which the matrix $B^s$ is rank one, which is a property shared by several physical systems. 
By Lemma \ref{lemma:rank1}, there exist $a \in \R$ and $p \in \R^n$ such that
\[
B^s = a p p^\top.
\]
In the sequel, we will work within the non-restrictive hypothesis $a = 1$. Furthermore, we will highlight the role played by the matrix $B^s$ in the product $X_k$, by simply splitting the latter (with a little abuse of notation) as
\[
X_k = B^s X_k.
\]
Moreover, we conventionally set
\[
X_{-1} = I,
\]
where $I$ is the identity matrix.
Then, in terms of the Euclidean inner product $\langle \cdot,\cdot \rangle $, \eqref{cancellation11} and \eqref{cancellation12} turn into
\begin{align}
\label{modcanc11}
    \bigl\langle B^s X_{k-1} A \partial_x U, \, - B^s X_{k-1} B^a U + m B^s X_{k-1} B^a A^2 U\bigl\rangle = 0 
    \end{align}
    and
    \begin{align}
\label{modcanc12}
    \bigl\langle B^s X_{k-1} A \partial_x U - m B^s X_{k-1} A B^a A\partial_x U, \, B^s X_{k-1} B^a U \bigl\rangle = 0,
\end{align}
while \eqref{cancellation21} and \eqref{cancellation22} become
\begin{equation}
    \label{modcanc21}
\bigl\langle B^s X_{k-1} U - m B^s X_{k-1}A^2 U, \, B^s X_{k-1} B^a A \partial_x U\bigl\rangle = 0,
\end{equation}
and
\begin{equation}
    \label{modcanc22}
\bigl\langle B^s X_{k-1} U - m B^s X_{k-1}AB^a U, \, B^s X_{k-1} B^a A \partial_x U\bigl\rangle = 0.
\end{equation}
We can now present the central result of this section.
\begin{lemma}
    Let $B^s$ be a rank-one, symmetric matrix.
    Then \eqref{modcanc11} is equivalent to
    \[\bigl\langle X_{k-1}p, A \partial_x U\bigl\rangle \bigl\langle B^a X_{k-1}p, (I - mA^2)U\bigl\rangle = 0,\]
    \eqref{modcanc12} is equivalent to
    \[\bigl\langle X_{k-1}p, B^a U\bigl\rangle \bigl\langle A X_{k-1}p, (I - mA^2)\partial_xU\bigl\rangle = 0.\]
    Similarly, \eqref{modcanc21} is equivalent to
    \[
    \bigl\langle B^aX_{k-1}p, A \partial_x U\bigl\rangle \bigl\langle X_{k-1}p, (I - mA^2)U\bigl\rangle = 0,
    \]
    and
    \eqref{modcanc22} is equivalent to
    \[
    \bigl\langle B^aX_{k-1}p, A \partial_x U\bigl\rangle \bigl\langle X_{k-1}p, (I - mAB^a)U\bigl\rangle = 0.
    \]
    \end{lemma}
\begin{proof}
    We look at the equivalence for \eqref{modcanc11}. The others are obtained in exactly the same way. Note that
    \[
    B^s X_{k-1} A \partial_x U = pp^\top X_{k-1} A \partial_x U = p (X_{k-1}^\top p)^\top A\partial_x U= p \bigl\langle X_{k-1}^\top p, A\partial_x U\bigl\rangle.
    \]
    By the same token,
    \[
     - B^s X_{k-1} B^a U + m B^s X_{k-1} B^a A^2 U = - p p^\top X_{k-1} B^a (I - mA^2) U = p \bigl\langle  B^a X_{k-1}^\top p,(I - mA^2) U\bigl\rangle.  
    \]
    However, $X_{k-1}$ is simply a product of $A$ and $B^a$. Hence, it holds
    \[
    X_{k-1}^\top = \pm X_{k-1},
    \]
    depending on how many times $B^a$ appears in the product. Since $X_{k-1}$ appears both in the left and in the right terms of the scalar product in \eqref{modcanc11}, the sign is actually irrelevant in this case. Recalling that $p$ is of unit norm, the proof is finished.
\end{proof}
We immediately infer the following corollary.
\begin{corollary}
    If 
    \begin{equation}
    \label{eq:suff11}
        p \in \ker(AX_{k-1}) \cup \ker((I-mA^2)B^aX_{k-1})
    \end{equation}
    and
    \begin{equation}
    \label{eq:suff12}
        p \in \ker(B^aX_{k-1}) \cup \ker((I-mA^2)AX_{k-1}),
    \end{equation}
    respectively, then the conditions \eqref{modcanc11} and \eqref{modcanc12} hold, respectively. Similarly, if
    \begin{equation}
    \label{eq:suff21}
        p \in \ker(AX_{k-1}B^a) \cup \ker((I-mA^2)X_{k-1}),
    \end{equation}
    then the condition \eqref{modcanc21} holds, while if 
    \begin{equation}
    \label{eq:suff22}
        p \in \ker(AX_{k-1}B^a) \cup \ker((I-mAB^a)X_{k-1}),
    \end{equation}
    condition \eqref{modcanc22} holds.
\end{corollary}
\begin{proof}
    The proof is a straightforward consequence of the fact that $A$ and $I - mA^2$ are symmetric.
\end{proof}
The conditions expressed above are purely algebraic. In general, we expect to be able to find $m$ such that the first holds. For that choice of $m$, the second condition translates into a requirement on the structural parameters of the equation. 

Note that, in theory, the kernels appearing in the second condition might be the trivial set $\{0\}$. In this case, there is no hope for an advantageous cancellation. It is then of interest to understand if, given a hyperbolic system, we can expect a condition of this kind. While a deeper investigation of this question is not the goal of this work, here we outline a connection between \eqref{modcanc11}-\eqref{modcanc12}, \eqref{modcanc21}-\eqref{modcanc22} and the spectrum of $A$. Let us look at the case of \eqref{modcanc11}, the other being symmetrical. Recall that, since $A$ is symmetric, we can orthogonally diagonalize it. To wit, we can find $Q$ such that $Q^{-1} = Q^\top$ and
\[
Q^\top A Q = D := \mathrm{diag}(\lambda_1, \ldots, \lambda_n),
\]
$\lambda_j$ being the eigenvalues of $A$.
We can further characterize the matrix $Q$ by noticing that
\[
Q = 
\bigg[\begin{array}{c|c|c|c}
v_1 &v_2 &\cdots &v_n
\end{array}\bigg],
\]
where $v_i$ are the normalized eigenvectors of $A$ relative to $\lambda_i$. Looking now at $B^s$, we have
\[
\Tilde{B}^s := Q^\top B^s Q =  Q^\top (p p^\top) Q =  (Q^\top p)(Q^\top p)^\top.
\]
In particular, $\tilde{B}^s$ is still rank one and symmetric. Calling $q = Q^\top p$, the matrix takes the form $\tilde{B}^s = qq^\top$. Similarly, we have that
\[
\Tilde{B}^a = Q^\top B^a Q,
\]
is still a skew-symmetric matrix.
Calling
\[
\Tilde{X}_k := Q^\top X_k Q,
\]
simple computations show that
 \[\bigl\langle X_{k-1}p, A \partial_x U\bigl\rangle \bigl\langle B^aX_{k-1}p, (I - mA^2)U\bigl\rangle = \bigl\langle \Tilde{X}_{k-1}q, D \partial_x V\bigl\rangle \bigl\langle \Tilde{B}^a\Tilde{X}_{k-1}q, (I - mD^2)V\bigl\rangle\]
and
\[
\bigl\langle B^aX_{k-1}p, A \partial_x U\bigl\rangle \bigl\langle X_{k-1}p, (I - mA^2)U\bigl\rangle =  \bigl\langle \Tilde{B}^a\Tilde{X}_{k-1}q, D \partial_x V\bigl\rangle \bigl\langle \Tilde{X}_{k-1}q, (I - mD^2)V\bigl\rangle,
\]
where $V = Q^\top U$. Let us consider, for simplicity, the case $\Tilde{X}_{k-1} = I$. Since \eqref{modcanc11} has to hold for every $U$ and $Q$ is invertible, the latter is satisfied if 
\[
q \in \ker(I-mD^2).
\]
Since $D$ is diagonal, the dimension of this kernel is larger if we can choose $m$ so that many of the entries of $m D^2$ are equal to $1$. In other words, the more eigenvalues $A$ has with the same absolute value, and the more "probable" it will be for $q$ to belong to this set.




\section{Applications}
\label{sec:applications}
\noindent
We now provide some motivating examples to show how the algorithm defined in the previous sections works in concrete situations. In what follows, we fix an admissible sequence $(p_k, q_k)$ in the sense of Definition \ref{admissible}. Furthermore, in line with Remark \ref{noLPHF}, we will work in the frequency variable. The reason for this is twofold: first of all, it makes the notation lighter. Secondly, it makes it easier to compare our results with the results already existing in the literature.

\subsection{A 3-by-3 System with cancellation}
\label{ex:3by3}

We begin with a simple 3-by-3 system, which can be seen as the coupling between a wave equation and a transport equation. Consider
\begin{equation}
\label{eq:3b3canc}
\begin{dcases}
    \partial_t u + a\partial_x v + w + u = 0, \\
    \partial_t v + a\partial_x u = 0, \\
    \partial_t w - u + b \partial_x w = 0.
\end{dcases}  
\end{equation}
We have
\[ A = 
\left[\begin{array}{ccc}
    0 & a &0  \\
    a & 0 &0 \\
    0 & 0 &b
\end{array} \right], \quad B^a = 
\left[\begin{array}{ccc}
    0 & 0 &1  \\
    0 & 0 &0 \\
    -1 & 0 &0
\end{array} \right]\quad \text{and }\quad
B^s = 
\left[\begin{array}{ccc}
    1 & 0 &0  \\
    0 & 0 &0 \\
    0 & 0 &0
\end{array} \right].
\]
It is straightforward to check that the inhomogeneous Kalman condition \eqref{condition:span} is satisfied.
The basic energy equality, obtained by rephrasing the system in Fourier space and computing the derivative of $\mathsf{E} = \tfrac{1}{2}|\widehat{U}|^2$, reads
\[
\partial_t \mathsf{E} + |\widehat{u}|^2 = 0.
\]
At this point, we split the analysis into high and low frequencies.

\medbreak
\noindent
\textbf{High Frequencies.} We use the algorithm of Section \ref{sec:4}. Observe that
\[
\rank\begin{bmatrix}
    B^s,\, B^s A
\end{bmatrix}^\top = 2 \quad \text{ and } \quad 
\rank\begin{bmatrix}
    B^s,\, B^s A,\, B^s B^a
\end{bmatrix}^\top = 3.
\]
Consequently, we fall into the mixed case. To apply Lemma \ref{lemma:magic}, we need to verify the assumptions. Since $\ell = 1$ we have $X_{\ell-1} = X_0 = B^s$. Assumption \ref{ass1} is verified since
\[
B^s A^2 U = a^2 u = a^2 B^s U, \quad \text{ and } \quad B^s A B^a U = 0.
\]
Besides,
\[
- B^s B^a U + m B^s B^a A^2 U = - u + m b^2 u = 0 \quad \iff \quad  m = \frac{1}{b^2}
\]
so Assumption \ref{ass2} is satisfied. Finally, 
\[
B^s B^a A B^a U = - b u = - b B^s U,
\]
and Assumption \ref{ass3} holds. At this point, we can check the conditions for cancellation. Since we are using Lemma \ref{lemma:magic}, this amounts to verifying \eqref{cancellation21}. Fixing $m = 1/b^2$, we have
\[
\bigl\langle B^s U - \frac{1}{b^2} B^s A^2 U, B^s B^a A\partial_x U\r = 
\bigl\langle u - \frac{a^2}{b^2}u, \partial_x w\r = 0 \quad \iff \quad a^2 = b^2. 
\]
If the coefficients of our system satisfy this equality, we can leverage cancellations to obtain a better rate of decay. According to the discussion in Section \ref{sec:7}:
\begin{itemize}
    \item If $a^2 \neq b^2$, we consider the Lyapunov functional
\begin{align*}
\L^h &= \frac{1}{2}|\widehat{U}|^2+  \frac{\eps}{|\xi|^2}{\rm Im} \l B^s  \widehat{U}, \xi B^s A \widehat{U}\r + \frac{\eps}{|\xi|^2}{\rm Re} \l B^s \widehat{U}, B^s B^a \widehat{U}\r + \frac{2\eps}{|\xi|^2} {\rm Re}\l B^s A \widehat{U}, B^s B^a A \widehat{U}\r \\
&=\frac{1}{2}|\widehat{U}|^2 +\frac{\eps}{|\xi|^2} {\rm Im}\l \widehat{u}, \xi\widehat{v}\r + \frac{a\eps}{|\xi|^2}{\rm Re} \l \widehat{u},\widehat{w}\r + \frac{2ab\eps}{|\xi|^2} {\rm Re}\l \widehat{v}, \widehat{w}\r.
\end{align*}
 \item If $a^2 = b^2$, we consider the Lyapunov functional
\begin{align*}
\L^h &= \frac{1}{2}|\widehat{U}|^2+ \frac{\eps}{|\xi|^2}{\rm Im} \l  B^s \widehat{U}, \xi B^sA \widehat{U}\r + \frac{\eps}{|\xi|^2}{\rm Re} \l B^s A \widehat{U}, B^s B^a A \widehat{U}\r \\ 
&\qquad\qquad\,\,+ \eps\sqrt{\eps}{\rm Re}\l B^s \widehat{U}, B^sB^a \widehat{U}\r + \eps\sqrt{\eps}{\rm Re} \l B^s A \widehat{U}, B^s B^a A \widehat{U}\r \\
&=\frac{1}{2}|\widehat{U}|^2 +\frac{\eps}{|\xi|^2}{\rm Im} \l\widehat{u}, \xi\widehat{v}\r + \frac{ab\eps}{|\xi|^2} {\rm Re}\l \widehat{v}, \widehat{w}\r + a\eps\sqrt{\eps}\, {\rm Re}\l \widehat{u},\widehat{w}\r + ab\eps\sqrt{\eps}\, {\rm Re}\l \widehat{v}, \widehat{w}\r.
\end{align*}
\end{itemize}
In the spirit of Theorem \ref{th:improvedh}, for $|\xi|\geq1$, the above construction with a suitably small $\varepsilon>0$ implies $|\widehat{U}|^2\lesssim e^{-c|\xi|^{-2}t}|\widehat{U}(0)|^2$ when $a^2\neq b^2$ and $|\widehat{U}|^2\lesssim e^{-ct}|\widehat{U}(0)|^2$ when $a^2=b^2$. Consequently, we arrive at
\[
\|U^h\|_{L^2} \leq C(1+t)^{-\gamma/2}\|U_0\|_{H^\gamma}, \quad \text{ if } a^2 \neq b^2
\]
and
\[
\|U^h\|_{L^2} \leq Ce^{-ct}\|U_0\|_{L^2}, \quad\quad  \text{ if } a^2 = b^2.
\]

\noindent
\textbf{Low Frequencies.}
In low frequencies, due to the presence of the mixed term, we cannot apply the theory of Section \ref{sec:ImproveLF}. Nevertheless, we can still use the general Kalman approach of Section \ref{sec:2}. By direct computations, the Lyapunov functional to consider in the low frequency regime is
\begin{align*}
\L^\ell&=\dfrac{1}{2}|\widehat{U}|^2+\eps\,{\rm Re}\l B^sU,B^s(\i \xi A+B^a)\widehat{U}\r+\eps\sqrt{\eps}\,{\rm Re}\l B^s(\i \xi A+B^a)\widehat{U},B^s(\i \xi A+B^a)^2\widehat{U}\r \\
&=\dfrac{1}{2}|\widehat{U}|^2+\eps\,{\rm Re}\l \widehat{u},i\xi a \widehat{v} + \widehat{w}\r+\eps\sqrt{\eps}\,{\rm Re}\l i\xi a \widehat{v} + \widehat{w},\widehat{u}+ |\xi|^2 a^2 \widehat{u} - i \xi b \widehat{w} \r.
\end{align*}
Theorem \ref{thm:decayKalman} then ensures that the solution decays. Although a general formula is not provided, in this simple case we can at least obtain an estimate for the decay rate. By Lemma \ref{lemma:equiv} and Lemma \ref{lemma:lowerbound}, this amounts to estimating the norm of the matrix $\widetilde{\mathsf{M}}^{-1}$, which for the system at hand reads
\[
\widetilde{\mathsf{M}}^{-1} =
\left[\begin{array}{ccc}
    1 & 0 &0  \\
    0 & i a \xi &1 \\
    -1 - a^2|\xi|^2 & 0 &i b \xi
\end{array} \right].
\]
Since by construction this matrix is non-singular, it is a well-known fact that the norm of the inverse is simply the inverse of the highest order singular value $\sigma_3(\widetilde{\mathsf{M}})$ (see also Remark \ref{rem:svd}). But the latter can also be characterized as the square root of the smallest eigenvalue of $\widetilde{\mathsf{M}}^\top \widetilde{\mathsf{M}}$. Hence, we compute the characteristic polynomial of this matrix, finding
\[
p(x) = -x^3
+ x^2 \left(3 + a^2 \xi^2 - b^2 |\xi|^2 + a^4 \xi^4\right) 
+ x \left(-2 + b^2 \xi^2 + a^4 \xi^4 - a^2 b^2 \xi^4 + a^6 \xi^6\right)
 + a^2 b^2 \xi^4.
\]
Finally, we observe that for every $\xi$ small enough, $p(a^2b^2\xi^4/4) > 0$ and $p(a^2b^2\xi^4) < 0$, meaning that $\sigma_3(\widetilde{\mathsf{M}}) \in (a^2b^2\xi^4/4, a^2b^2\xi^4)$. We conclude that $\beta$, appearing in \eqref{decay:beta} and thus in Theorem \ref{thm:decayKalman}, is equal to 2. This can also be verified numerically, computing the eigenvalues of $i \xi A$ for small values of $\xi$.

\subsection{The Timoshenko System}
Next, we turn to the Timoshenko system
\begin{equation}\label{eq:timo0}
    \begin{dcases}
    \partial_{tt}^2\varphi-\partial_x\big(\partial_x\varphi-\psi\big)=0,\\
    \partial_{tt}^2 \psi-\partial_x\big( \sigma(\partial_x\varphi) \big)-(\partial_x\varphi-\psi)+b\partial_t\varphi=0
        \end{dcases}
\end{equation}
where the smooth function $\sigma$ satisfies $g'_*(s)>0$ for any $s\in\mathbb{R}$. System \eqref{eq:timo0} describes the transverse vibrations of a beam; see \cite{timo1,timo2} and see \cite{IHK} and references therein its mathematical analysis.

Defining the sound speed $a^2=\sigma(0)$ and the variables
\begin{align}\label{charge}
u_1=\partial_x\varphi-\psi,\quad u_2=\partial_t \varphi,\quad u_3=a \partial_x\psi,\quad u_4=\partial_t \psi,
\end{align}
the system \eqref{eq:timo0} rewrites
\begin{equation}
    \label{eq:timo}
    \begin{dcases}
        \partial_t u_1 - \partial_x u_2 + u_4 = 0, \\
        \partial_t u_2 - \partial_x u_1 = 0, \\
        \partial_t u_3 - a \partial_x u_4 = 0, \\
        \partial_t u_4 - a \partial_x u_3 - u_1 + b u_4 =\sigma(\frac{u_3}{a})-\sigma(0)-\sigma'(0)\frac{u_3}{a}.
    \end{dcases}
\end{equation}
The linear system for \eqref{eq:timo} can be written in the form \eqref{eq:1-1} with
\[
A =
\begin{bmatrix}
   0 &-1 &0 &0 \\
   -1 &0 &0 &0 \\
   0 &0 &0 &-a \\
   0 &0 &-a &0
\end{bmatrix}, \quad
B^a =
\begin{bmatrix}
   0 &0 &0 &1 \\
   0 &0 &0 &0 \\
   0 &0 &0 &0 \\
   -1 &0 &0 &0
\end{bmatrix}, \quad
B^s =
\begin{bmatrix}
   0 &0 &0 &0 \\
   0 &0 &0 &0 \\
   0 &0 &0 &0 \\
   0 &0 &0 &b
\end{bmatrix}.
\]
We indicate with $U$ the solution vector $(u_1,u_2,u_3,u_4)$ and work in the Fourier space.
To initialize the algorithm, we compute the derivative of the Fourier energy $\mathsf{E} = \tfrac{1}{2}|\widehat{U}(t,\xi)|^2$, obtaining the basic energy equality
\[
\partial_t \mathsf{E} + b |\widehat{u}_4|^2 = 0.
\]
At this point, setting the first node $X_0 = B^s$, we check if one of the conditions \eqref{eq:rankleft}-\eqref{eq:rankmix2} holds. For the system at hand, we have
\[
\rank\begin{bmatrix}
    B^s,\, B^s A
\end{bmatrix}^\top = 2 \quad \text{ and } \quad 
\rank\begin{bmatrix}
    B^s,\, B^s A,\, B^s B^a
\end{bmatrix}^\top = 3,
\]
so that \eqref{eq:rankmix} holds and we fall into the mixed case. In order to apply Lemma \ref{lemma:magic}, we need to verify the three assumptions. Since
\[
B^sA^2 = a^2 B^s, \text{ and } B^sAB^a = B^sB^aAB^a = 0.
\]
Assumptions \ref{ass1} and \ref{ass3} are satisfied.
As for Assumption \ref{ass2}, we need to find $m$ so that
\[
\l B^s A \partial_x U, -B^sB^aU + m B^s B^a A^2 U\r = a b^2 \l \partial_x u_1, u_3\r - m a b^2 \l \partial_x u_1, u_3\r = 0.
\]
Clearly, this is satisfied for $m=1$. Before writing down the functional, we look for cancellation conditions with $m=1$ now fixed. Condition \eqref{cancellation21} reads
\[
- b^2 \l u_2, \partial_x u_4\r + a^2 b^2 \l u_2, \partial_x u_4\r = 0,
\]
and it is satisfied if $a^2 = 1$. This is precisely the equal wave speed condition that can be found in \cite{IHK}. In conclusion, we add the nodes $X_1^1 = B^s A$ and $X_1^2 = B^s B^a$. Since the discrepancy of the first node $B^s$ is set conventionally to zero, we obtain the functionals
\[
\Psi_1^1 = \frac{\eps}{|\xi|^2}\left({\rm Im}\l B^s \widehat{U}, \xi B^s A\widehat{U}\r + {\rm Re}\l B^s A \widehat{U}, B^sB^a A \widehat{U}\r\right) = \frac{\eps a b^2}{|\xi|^2} {\rm Im} \l \widehat{u}_4, \xi\widehat{u}_3\r - \frac{\eps a b^2}{|\xi|^2}{\rm Re} \l \widehat{u}_3, \widehat{u}_2\r,
\]
and
\[\Psi_1^2 =
\begin{dcases}
    \frac{\eps}{|\xi|^2}{\rm Re}\l B^s \widehat{U}, B^sB^a \widehat{U}\r + \frac{\eps}{|\xi|^2}{\rm Re}\l B^s A \widehat{U}, B^s B^a A \widehat{U}\r \\ 
    \qquad= -\frac{\eps b^2}{|\xi|^2}{\rm Re} \big(\l \widehat{u}_4,\widehat{u}_1\r + a \l \widehat{u}_3,\widehat{u}_2\r\big) 
    &\text{if } a^2 \neq 1,  \\ 
    \eps^{p_2}{\rm Re}\l B^s \widehat{U}, B^s B^a \widehat{U}\r + \eps^{p_2}{\rm Re} \l B^s A \widehat{U}, B^s B^a A \widehat{U}\r\\ 
    \qquad= -\eps^{p_2}b^2{\rm Re} \big(\l \widehat{u}_4,\widehat{u}_1\r + a \l \widehat{u}_3,\widehat{u}_2\r\big) &\text{if } a^2 = 1.
\end{dcases}
\]
We now move to the next step. It is not difficult
to see that the only way to improve the rank is to add the node $X_2^1 = B^s B^a A$. Indeed, 
\[
\rank\begin{bmatrix}
    B^s,\, B^s A,\, B^s B^a, \, B^s B^a A
\end{bmatrix}^\top = 4.
\]
The functional here depends on whether $a^2 = 1$ or not. In the cancellation case, the discrepancy of the node $B^sB^a$ is still equal to zero, otherwise it is 1. Accordingly, we get
\[
\Psi_2^1 = 
\begin{dcases}
    \frac{\eps^{p_2} }{|\xi|^4}{\rm Im}\l B^s B^a \widehat{U}, \xi B^s B^a A \widehat{U}\r = \frac{\eps^{p_2}b^2 }{|\xi|^4} {\rm Im}\l \widehat{u}_1, \xi\widehat{u}_{2}\r &\text{if } a^2 \neq 1,  \\ 
    \frac{\eps^{p_3}}{|\xi|^2}{\rm Im}\l B^s B^a \widehat{U},\xi B^s B^a A \widehat{U}\r = \frac{\eps^{p_3}b^2}{|\xi|^2}{\rm Im} \l \widehat{u}_1, \xi \widehat{u}_2\r &\text{if } a^2 = 1.
\end{dcases}
\]
Since the rank is now equal to the number of variables, we can stop. For $a^2\neq1$, the complete Lyapunov functional reads
\begin{align*}
\L^h = &\frac{1}{2}|\widehat{U}|^2 + \frac{\eps a b^2}{|\xi|^2} \big( \xi {\rm Im}\l \widehat{u}_4, \widehat{u}_3\r - {\rm Re}\l \widehat{u}_3, \widehat{u}_2\r \big) \\
&-\frac{\eps b^2}{|\xi|^2} {\rm Re}\big(\l \widehat{u}_4,\widehat{u}_1\r + a \l \widehat{u}_3,\widehat{u}_2\r\big) + \frac{\eps^{p_2}b^2}{|\xi|^4} {\rm Im}\l \widehat{u}_1, \xi\widehat{u}_{2}\r,
\end{align*}
 On the other hand, when $a^2 = 1$,
\begin{align*}
\L^h = &\frac{1}{2}|\widehat{U}|^2 + \frac{\eps a b^2}{|\xi|^2} \big( \xi{\rm Im} \l \widehat{u}_4, \widehat{u}_3\r - {\rm Re}\l \widehat{u}_3, \widehat{u}_2\r \big) \\
&-\eps^{p_2}b^2 \big(\l \widehat{u}_4,\widehat{u}_1\r + a \l \widehat{u}_3,\widehat{u}_2\r\big) +\frac{\eps^{p_3}b^2}{|\xi|^2} {\rm Im}\l \widehat{u}_1,\xi \widehat{u}_2\r.
\end{align*}
Thus, using Theorem \ref{th:improvedh}, we get the properties
\[
\L^h\sim|\widehat{U}|^2 ,\quad  \frac{d}{dt}\L^h + \frac{c}{\xi^{2\alpha}}\L^h \leq 0,
\]
where $\alpha = 1$ if $a^2 \neq 1$ and $\alpha = 0$ if $a^2 = 1$. In turn, for any $\gamma>0$, this yields the usual decay rate of the Timoshenko system in high frequencies \cite{IHK}:
\[
\|U^h\|_{L^2} \leq C(1+t)^{-\gamma/2}\|U_0\|_{H^{\gamma\alpha}}, \quad \text{ if } a^2 \neq 1,
\]
and
\[
\|U^h\|_{L^2} \leq Ce^{-ct}\|U_0\|_{L^2},\quad  \quad \text{ if } a^2 = 1.
\]
\begin{remark}
\label{rem:rankone}
    Calling 
    \[
    p = [0, 0, 0, 1],
    \]
    we have 
    \[
    B^s = \gamma pp^\top.
    \]
    In particular, $B^s$ is rank one, so we fall into the framework of Section \ref{sec:8}. Note that 
    \[
    (I - m A^2)B^a = 
    \begin{bmatrix}
    0 &0 &0 &1-m \\
    0 &0 &0 &0 \\
    0 &0 &0 &0 \\
    -1+ma^2 &0 &0 &0
\end{bmatrix}.
    \]
    Hence, unless $m$ is equal to $1$, $p$ does not belong to its kernel and \eqref{eq:suff11} is not satisfied (as $p$ certainly does not belong to $\ker(A)$). Once we set $m=1$ we see that \eqref{eq:suff21} is satisfied only if $a^2 = 1$. 
\end{remark}
\subsection{The Timoshenko System with Memory} In \cite{MoriMem}, the following Timoshenko System with Memory is considered:
\begin{equation}\label{eq:timomem0}
    \begin{dcases}
    \partial_{tt}^2\varphi-\partial_x\big(\partial_x\varphi-\psi\big)=0,\\
    \partial_{tt}^2 \psi-\partial_x\big( g(\partial_x\varphi) \big)-(\partial_x\varphi-\psi)+b\partial_t\varphi+b g\ast \partial_{xx}^2 \psi =0,
        \end{dcases}
\end{equation}
where the term $ g\ast \partial_{xx}^2 \psi =\int_0^t g(t-\tau) \partial_{xx}^2 \psi(\tau)d\tau$ is referred to as a memory-type damping. Taking $g(t)=\mu e^{-\mu t}$ with $\mu>0$ and rewriting \eqref{eq:timomem0} as a first-order system, its linearization around a steady solution reads
\begin{equation}
    \label{eq:timomem}
    \begin{dcases}
        \partial_t u_1 - \partial_x u_3 + u_2 = 0, \\
        \partial_t u_2 - c_1 \partial_x u_4 + c_2 \partial_x u_5 - u_1 = 0, \\
        \partial_t u_3 - \partial_x u_1 = 0, \\
        \partial_t u_4 - c_1 \partial_x u_2 = 0, \\
        \partial_t u_5 + c_2 \partial_x u_2 + \mu u_5 = 0.
    \end{dcases}
\end{equation}
This equations are derived from the Timoshenko system \eqref{eq:timo}, when memory effect on the rotation angle are present.
The matrices are
\[
A =
\begin{bmatrix}
   0 &0 &-1 &0 &0 \\
   0 &0 &0 &-c_1 &c_2 \\
   -1 &0 &0 &0 &0 \\
   0 &-c_1 &0 &0 &0 \\
   0 &c_2 &0 &0 &0
\end{bmatrix}, \quad
B^a =
\begin{bmatrix}
   0 &1 &0 &0 &0 \\
   -1 &0 &0 &0 &0 \\
   0 &0 &0 &0 &0 \\
   0 &0 &0 &0 &0 \\
   0 &0 &0 &0 &0
\end{bmatrix}, \quad
B^s =
\begin{bmatrix}
   0 &0 &0 &0 &0 \\
   0 &0 &0 &0 &0 \\
   0 &0 &0 &0 &0 \\
   0 &0 &0 &0 &0 \\
   0 &0 &0 &0 &\mu
\end{bmatrix}.
\]
We denote $U=(u_1,u_2,u_3,u_4,u_5)$. Since it will not play a role in the subsequent computations, we set $\mu = 1$. 
After the basic energy equality
\[
\partial_t \mathsf{E} + |\widehat{u}_5|^2 = 0,
\]
with $ \mathsf{E}=\frac{1}{2}|\widehat{U}|^2$, we begin by moving left, adding the node $X_1^1 = B^sA$. Indeed,
\[
\rank\begin{bmatrix}
    B^s,\, B^s A
\end{bmatrix}^\top = 2, \quad \text{ and } \quad
\rank\begin{bmatrix}
    B^s,\, B^s B^a
\end{bmatrix}^\top = 1.
\]
Hence, we obtain
\[
\Psi_1^1 = \frac{\eps}{|\xi|^2}{\rm Im}\l  B^s \widehat{U}, \xi B^s A \widehat{U}\r = \frac{\eps c_2 }{|\xi|^2}{\rm Im} \l \widehat{u}_5, \xi\widehat{u}_{2} \r.
\]
Moving to the second level, we notice that we are in the mixed path situation. It is not difficult to check that Assumptions \ref{ass1}-\ref{ass3} hold and that Assumption \ref{ass2} is satisfied if $m=1$.
Substituting this into \eqref{cancellation21} we find that there is a cancellation if 
\[
c_1^2 + c_2^2 = 1,
\]
which is exactly the cancellation condition in \cite{MoriMem}. Note that even in this case we could have worked as in Remark \ref{rem:rankone}, since $B^s$ is a rank one matrix. We add the nodes $X_2^1 = B^s A^2$ and $X_2^2 = B^s A B^a$ and obtain the functionals
\[
\Psi_2^1 = \frac{\eps^{p_2} c_2^2 }{|\xi|^2}{\rm Im} \l \xi\widehat{u}_2, -c_1 \widehat{u}_{4} + c_2 \widehat{u}_{5}\r + \frac{\eps^{p_2} c_2^2 }{|\xi|^2}{\rm Re} \l \widehat{u}_3,  -c_1 \widehat{u}_{4} + c_2 \widehat{u}_{5}\r,
\]
and
\[
\Psi_2^2 = 
\begin{dcases}
    -\frac{\eps^{p_2} c_2^2}{|\xi|^2}{\rm Re}\l \widehat{u}_2, \widehat{u}_1\r +\frac{\eps^{p_2} c_2^2 }{|\xi|^2}{\rm Re} \l \widehat{u}_3,  -c_1 \widehat{u}_{4} 
    + c_2 \widehat{u}_{5}\r &\text{if } c_1^2 + c_2^2 \neq 1, \\
   -\eps^{p_3} c_2^2{\rm Re}\l \widehat{u}_2, \widehat{u}_1\r + \eps^{p_3} c_2^2 {\rm Re} \l \widehat{u}_3,  -c_1 \widehat{u}_{4} + c_2 \widehat{u}_{5}\r &\text{if } c_1^2 + c_2^2 = 1.
\end{dcases}
\]
Finally, we observe that the last node needed is $X_3^1 = B^s A B^a A$, yielding
\[
\Psi_3^1 = 
\begin{dcases}
    -\frac{\eps^{p_3} c_2^2}{|\xi|^4}{\rm Im}\l \widehat{u}_1, \xi\widehat{u}_{3}\r &\text{if } c_1^2 + c_2^2 \neq 1, \\
    -\frac{\eps^{p_4} c_2^2}{|\xi|^2} {\rm Im}\l \widehat{u}_1, \xi\widehat{u}_{3}\r  &\text{if } c_1^2 + c_2^2 = 1.
\end{dcases}
\]
The complete Lyapunov functional is then, if $c_1^2 + c_2^2 \neq 1$,
\begin{align*}
\L^h = &\frac{1}{2}|\widehat{U}|^2 + \frac{\eps c_2 }{|\xi|^2}{\rm Im} \l \widehat{u}_5,\xi\widehat{u}_{2} \r +\frac{\eps^{p_2} c_2^2 }{|\xi|^2} {\rm Im}\l \widehat{u}_2, -c_1 \xi \widehat{u}_{4} + c_2 \xi\widehat{u}_{5}\r \\
&+ \frac{2\eps^{p_2} c_2^2}{|\xi|^2}{\rm Re} \l \widehat{u}_3,  -c_1 \widehat{u}_{4} + c_2 \widehat{u}_{5}\r - \frac{\eps^{p_2} c_2^2}{|\xi|^2}{\rm Re}\l \widehat{u}_2, \widehat{u}_1\r - \frac{\eps^{p_3} c_2^2 }{|\xi|^4}{\rm Im}\l \widehat{u}_1, \xi\widehat{u}_{3}\r
\end{align*}
 and in the cancellation case $c_1^2 + c_2^2 = 1$,
\begin{align*}
\L^h &=  \frac{1}{2}|\widehat{U}|^2 +\frac{\eps c_2 }{|\xi|^2}{\rm Im} \l \widehat{u}_5, \xi\widehat{u}_{2} \r +\frac{\eps^{p_2} c_2^2 }{|\xi|^2} {\rm Im}\l \widehat{u}_2, -\xi c_1 \widehat{u}_{4} + c_2\xi\widehat{u}_{5}\r + \frac{\eps^{p_2} c_2^2}{|\xi|^2}{\rm Re} \l \widehat{u}_3,  -c_1 \widehat{u}_{4} + c_2 \widehat{u}_{5}\r \\
&- \eps^{p_3} c_2^2{\rm Re}\l \widehat{u}_2, \widehat{u}_1\r+ \frac{\eps^{p_3} c_2^2}{|\xi|^2}{\rm Re} \l \widehat{u}_3,  -c_1 \widehat{u}_{4} + c_2 \widehat{u}_{5}\r -\frac{\eps^{p_4} c_2^2 }{|\xi|^2}{\rm Im}\l \widehat{u}_1, \xi\widehat{u}_{3}\r.
\end{align*}
Then, following using Theorem \ref{th:improvedh}, we obtain
\[
\L^h\sim |\widehat{U}^h|,\quad \frac{d}{dt}\L^h + \frac{c}{\xi^{2\alpha}}\L^h \leq 0,
\]
where $\alpha = 1$ if there is no cancellation and $\alpha = 0$ in the case of cancellation. This leads to the decay rates
\[
\|U^h\|_{L^2} \leq C(1+t)^{-\gamma/2}\|U_0\|_{H^{\gamma\alpha}}, \quad \text{ if }\quad c_1^2+c_2^2 \neq 1,
\]
and
\[
\|U^h\|_{L^2} \leq Ce^{-ct}\|U_0\|_{L^2},\quad  \quad \text{ if }\quad c_1^2+c_2^2 = 1.
\]
Such decay rates are predicted by the spectral analysis.



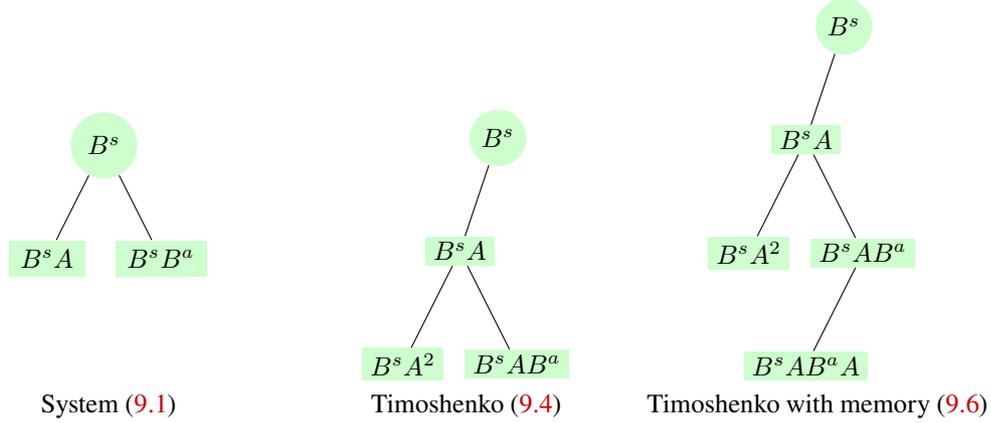
\begin{figure}[h!]
\begin{minipage}[t]{0.33\linewidth}
    \centering
     \begin{tikzpicture}[scale=1]
            [level distance=15mm,
        every node/.style={fill=green!20,inner sep=1pt},
        level 1/.style={sibling distance=10mm,nodes={fill=red!0}},
        level 2/.style={sibling distance=15mm,nodes={fill=red!20}}]
    \node[fill=green!20,circle] {$B^s$}
        child{node[fill=green!20]{$B^sA$}}
        child{node[fill=green!20]{$B^sB^a$}
            child[edge from parent/.style={draw=none}] {node[draw=none]{}}};
        \end{tikzpicture}\\
        System \eqref{eq:3b3canc}
\end{minipage}
\begin{minipage}[t]{0.33\linewidth}
        \centering
        \begin{tikzpicture}
            [level distance=15mm,
        every node/.style={fill=green!20,inner sep=2pt},
        level 1/.style={sibling distance=10mm,nodes={fill=red!0}},
        level 2/.style={sibling distance=15mm,nodes={fill=red!20}}]
    \node[fill=green!20,circle] {$B^s$}
        child{node[fill=green!20]{$B^sA$}
            child{node[fill=green!20]{$B^sA^2$}}
            child{node[fill=green!20]{$B^sAB^a$}
          }}
        child[edge from parent/.style={draw=none}] {node[draw=none]{}};
        \end{tikzpicture}\\
        Timoshenko \eqref{eq:timo}
    \end{minipage}%
    \begin{minipage}[t]{0.33\linewidth}
        \centering
        \begin{tikzpicture}
        [level distance=15mm,
        every node/.style={fill=green!20,inner sep=2pt},
        level 1/.style={sibling distance=10mm,nodes={fill=red!0}},
        level 2/.style={sibling distance=15mm,nodes={fill=red!0}}]
    \node[fill=green!20,circle] {$B^s$}
        child{node[fill=green!20]{$B^sA$}
            child{node[fill=green!20]{$B^sA^2$}}
            child{node[fill=green!20]{$B^sAB^a$}
            child{node[fill=green!20]{$B^sAB^aA$}}
            child[edge from parent/.style={draw=none}] {node[draw=none]{}}  }}
        child[edge from parent/.style={draw=none}] {node[draw=none]{}};
        \end{tikzpicture}\\
        Timoshenko with memory \eqref{eq:timomem}
    \end{minipage}
    \caption{The paths followed in the three examples.}
    \label{fig:timo}
\end{figure}

\section{Conclusions and Future Work}
\label{sec:10}

\noindent
In this work, we have introduced a new Kalman-type stability condition for hyperbolic systems with non-symmetric relaxation. Using this condition, we have shown that the solutions decay at a certain rate as $t \to \infty$. To identify this rate, we have proposed an algorithm to construct a Lyapunov functional. This strategy can be applied to any linear hyperbolic system, to systematically estimate the rate of decay of the solution, and to uncover advantageous algebraic cancellations in the high-frequency regime.

We highlight two potential research directions which remain to be explored. 
\begin{enumerate}
    \item The first crucial question concerns the optimal decay rate. The primary reason why our algorithm may not provide the optimal decay rate is the choice \ref{alwaysleft}, which makes the chosen path on the tree unique. One could theoretically list all the possible paths on the tree leading to the recovery of the dissipation of every involved variable, and select the one with the best decay rate. On the other hand, finding conditions on $A,B^a$ and $B^s$ to ensure that the chosen path yields the optimal decay rate is an interesting open problem.

    \item It would be equally interesting to extend our framework to partially diffusive systems, namely to equations of the form
    \[
    \partial_t U + A \partial_x U + B U - C \partial_{xx}U = 0,
    \]
    where $C$ is a symmetric matrix whose rank is not full.
   In this setting, four operators, two dissipative and two conservative, would interact together, rendering the analysis more intricate.  Nonetheless, it is possible to extend the inhomogeneous Kalman condition to this setting and design a procedure to retrieve the decay rates of such systems following the methodology developed in the present paper.
\end{enumerate}
\vspace{1cm}
\section*{Acknowledgments}
\noindent
T. Crin-Barat is supported by the project ANR-24-CE40-3260 -- Hyperbolic Equations, Approximations \& Dynamics (HEAD). L. Liverani is supported by the Alexander von Humboldt Fellowship for Postdoctoral Researchers. L.-Y. Shou is supported by the National Natural Science Foundation of China (12301275). E. Zuazua has been funded by the Alexander von Humboldt-Professorship program and the Transregio 154 Project “Mathematical Modelling, Simulation and Optimization Using the Example of Gas Networks” of the DFG, the ModConFlex Marie Curie Action, HORIZON-MSCA-2021-DN-01, the COST Action MAT-DYN-NET, grants PID2020-112617GB-C22 and TED2021-131390B-I00 of MINECO (Spain), and by the Madrid Government -- UAM Agreement for the Excellence of the University Research Staff in the context of the V PRICIT (Regional Programme of Research and Technological Innovation).

\vspace{1cm}
\appendix

\section{Linear Algebra results on Rank-one Matrices}\label{sec:appendixB}

 We begin by recalling basic facts.
\begin{lemma}
\label{lemma:rank1}
    The matrix $A\in \mathbb{M}(n,n)$ is a rank-one symmetric matrix if and only if there exists $a \in \R$ and $p\in \R^n$ of unitary norm such that
    \[
    A = app^\top,
    \]
    where $pp^\top$ is the outer product of $p$ with itself.
\end{lemma}

\begin{lemma}
\label{lemma:structure}
    Let $A \in \M(n,n)$ be a symmetric, rank-one matrix. Then for any $B \in \M(n,n)$, $AB$ is also rank-one.
\end{lemma}
\begin{proof}
    Since $A$ is rank-one, there exist $a \in \R$, $p \in \R^n$ such that $A = a p p^\top$. Without loss of generality, let us assume $a=1$. Then,
    \[
    A = \left[\begin{array}{c}
         p_1 p^\top\\
         p_2 p^\top \\
         \ldots \\
         p_n p^\top
    \end{array}\right].
    \]
    Consequently, it follows that
    \[
    A B = 
    \left[\begin{array}{c}
         p_1 p^\top B\\
         p_2 p^\top B \\
         \ldots \\
         p_n p^\top B
    \end{array}\right].
    \]
    Calling $w:=p^\top B$, we see that $AB$ is also rank-one.
\end{proof}
\begin{lemma}
    Let $X \in \M(p,n)$, $p \geq n$ be such that $\rank \, X = k$ and let $A \in \M(n,n)$ be a symmetric, rank-one matrix. Then, for any $B \in \M(n,n)$, we have
    \[
    k \leq \rank\left([X, AB]^\top\right) \leq k+1.
    \]
\end{lemma}
\begin{proof}
    By the previous lemma, $AB$ is a rank-one matrix. Therefore, the dimension of the row space of $AB$ is equal to 1. At this point, two possibilities can occur: $\mathrm{row}(AB) \subset \mathrm{row}(X)$, and in this case $\rank([X,AB]^\top) = k$, or $\mathrm{row}(AB) \cap \mathrm{row}(X) = \{0\}$, and the rank is equal to $k+1$.
\end{proof}
For the next lemma, we need a definition.
\begin{definition}
    Let $P_0, \ldots, P_d$ be matrices in $\M(n,n)$, with $P_d \neq 0$. We call
    \[
    P(x) = P_0 + P_1 x + \ldots + P_d x^d
    \]
    a polynomial matrix of degree $d$.
\end{definition}
\begin{lemma}
\label{lemma:poly}
    Let $P(x)$ be a polynomial matrix of degree $d$, and suppose there exists a constant $\eta > 0$ such that $P(x)$ is invertible for every $x > \eta$. Then there exists $\alpha$ and a $C>0$ such that
    \[
    \|P^{-1}(x)\| \leq C (1+x)^\alpha.
    \]
\end{lemma}

\begin{proof}
    Let $p(\lambda,x) = \det(P(x) - \lambda I)$ be the characteristic polynomial of $P(x)$. We can consider $P$ as a matrix over the commutative ring $\R[x]$, so that the Cayley-Hamilton theorem holds. Accordingly, 
    \[
    p(P(x),x) = 0.
    \]
    Since $P(x)$ is invertible for every $x > \eta$, we have that $p(0,x) = \det(P(x))$ is a polynomial with no zeros for $x > \eta$. Hence,
    \[
    \frac{1}{p(0,x)} \lesssim 1.
    \]
    Finally, writing
    \[
    p(\lambda,x) = \lambda r(\lambda,x) + p(0,x),
    \]
    for some other polynomial $r$, we obtain
    \[
    P^{-1}(x) = -\frac{r(P(x),x)}{p(0,x)}.
    \]
    Since $r(P(x),x)$ has polynomial growth, the sought inequality follows.
\end{proof}
\begin{remark}
\label{rem:complex}
    Lemma \ref{lemma:poly} effortlessly extends to polynomial matrices with coefficients in $\mathbb{C}^{n\times n}$. 
\end{remark}
\begin{lemma}
\label{lemma:lowerbound}
    Let $P(x)$ be a polynomial matrix such that $P(x)$ is invertible for every $x$ large enough. Let $\alpha = \alpha(P)$ be the exponent in the previous lemma. There exists a $c>0$ such that
    \[
    \|P(x)U\| \geq \frac{c}{x^\alpha} \|U\|.
    \]
\end{lemma}

\begin{proof}
    The proof follows by simply noting that, for every invertible matrix $A$,
    \[
    \|U\| = \|A^{-1}AU\| \leq \|A^{-1}\| \|AU\|.
    \]
    An application of the Lemma \ref{lemma:poly} yields the statement.
\end{proof}

\begin{remark}
\label{rem:svd}
    Unfortunately, obtaining a non-vanishing lower bound is generally not possible. A simple example is the polynomial matrix
    \[
    P(x) = \left[ \begin{array}{cc}
        1 &x  \\
        0 &1 
    \end{array}\right].
    \]
    Indeed, while $P(x)$ is invertible for every $x$, it is possible to show that its smallest singular value
    \[
    \sigma_2(P(x)) := \min_{U \in \R^2, \|U\|=1} \|P(x)U\| \sim \frac{1}{x^2}, \quad \text{ as } x \to + \infty.
    \]
    Recalling that the norm of the inverse can be equivalently characterized as $(\sigma_2(P(x)))^{-1}$ \cite[Theorem 3.3]{algebra_book}, we see that in this case the estimate in Lemma \ref{lemma:poly} cannot be improved.
\end{remark}

\section{Proof of Lemma \ref{lemma:equiv}}
\label{sec:appproof}

\noindent
Let us work in high frequencies. The low-frequency case works in the same fashion. For every nonzero $x \in \R$, we define the matrix
\[
\mathsf{M}(\xi) = \left[B^s, \frac{1}{|\xi|}B^s(\i \xi A + B^a), \ldots, \frac{1}{|\xi|^K}B^s (\i \xi A + B^a)^K \right]^\top \in \mathbb{C}^{Kn \times n}.
\]
We recall that the rank of a matrix can be equivalently defined as
\[
\rank\, \mathsf{M}(\xi) = \dim \mathrm{row} \, \mathsf{M}(\xi).
\]
In particular, the inhomogeneous Kalman condition holds for $\mathsf{M}(\xi)$ and it asserts that the dimension of the row space of $\mathsf{M}(\xi)$ is equal to $n$. Equivalently, there exist $n$ rows $v_1(\xi), \ldots, v_n(\xi)$ of 
$\mathsf{M}(\xi)$ which are linearly independent. Hence, we have
\[
|\mathsf{M}(\xi)\widehat{U}|^2 = \sum_{k=0}^{K}|\xi|^{-2k}|B^s(\i \xi A+B^a)^k \widehat{U}|^2 \geq \sum_{j=0}^n |\l v_j(\xi), \widehat{U}\r|^2 = | \Tilde{\mathsf{M}}(\xi)\widehat{U}|^2,
\]
where $\Tilde{\mathsf{M}}(\xi)$ is the matrix
\[
\Tilde{\mathsf{M}}(\xi) = \left[v_1(\xi), \ldots, v_n(\xi)\right]^\top \in \mathbb{C}^{n \times n}.
\]
Clearly $\Tilde{\mathsf{M}}(\xi)$ is invertible. Besides, it is apparent that it is a polynomial matrix in $1/\xi$. We can thus rewrite it as
\[
\Tilde{\mathsf{M}}(\xi) = \frac{1}{|\xi|^{\alpha'}}
\Tilde{\mathsf{M}}'(\xi),
\]
where $\alpha'\in\mathbb{N}$ is the degree of $\Tilde{\mathsf{M}}$, and $\Tilde{\mathsf{M}}'$ is an invertible polynomial matrix in $\xi$. 
By Lemma \ref{lemma:lowerbound}, we have
\[
|\Tilde{\mathsf{M}}'\widehat{U}(\xi)|\geq \frac{1}{|\xi|^{\alpha''}}|\widehat{U}(\xi)|.
\]
The thesis follows with $\alpha = \alpha' + \alpha''$. \qed

\bibliographystyle{apalike}  
\bibliography{NonSym}
\vfill

\end{document}